\documentclass[12pt,reqno]{amsart}
\usepackage{mathrsfs}
\usepackage{amsfonts}
\usepackage{amsmath, amssymb}
\usepackage{setspace}
\usepackage{indentfirst}
\usepackage{enumerate, hyperref}
\usepackage{dcolumn}
\numberwithin{equation}{section}

\newtheorem{thm}{Theorem}[section]
\newtheorem{lem}[thm]{Lemma}
\newtheorem{prop}[thm]{Proposition}

\theoremstyle{definition}
\newtheorem{rem}[]{Remark}
\newcommand\R{{\mathbb R}}
\allowdisplaybreaks

\begin{document}
\title[3D inhomogeneous incompressible Navier-Stokes Systems]{Global axisymmetric solutions of 3D inhomogeneous incompressible Navier-Stokes Systems with nonzero swirl}
\author{Hui Chen, Daoyuan Fang, Ting Zhang}

\begin{abstract}
In this paper, we investigate the global well-posedness for the 3-D inhomogeneous incompressible Navier-Stokes system with the axisymmetric initial data. We prove the global well-posedness  provided that
 $$\|\frac{a_{0}}{r}\|_{\infty} \textrm{ and } \|u_{0}^{\theta}\|_{3}  \textrm{ are sufficiently small}.
   $$
 Furthermore, if $\mathbf{u}_0\in L^1$ and $ru^\theta_0\in L^1\cap L^2$, we have
\begin{equation*}
\|u^{\theta}(t)\|_{2}^{2}+\langle t\rangle \|\nabla (u^{\theta}\mathbf{e}_{\theta})(t)\|_{2}^{2}+t\langle t\rangle(\|u_{t}^{\theta}(t)\|_{2}^{2}+\|\Delta(u^{\theta}\mathbf{e}_{\theta})(t)\|_{2}^{2}) \leq C \langle t\rangle^{-\frac{5}{2}},\ \forall\ t>0.
\end{equation*}
\end{abstract}
\maketitle
\section{Introduction}

In this paper, we consider the initial value problem of 3D inhomogeneous incompressible Navier-Stokes equations with the axisymmetric initial data:
\begin{equation}\label{1.1}
\left\{
\begin{array}{l}
\partial_{t}\rho+ div(\rho \mathbf{u})=0,        ~(t,x)\in\R^+\times\R^{3},\\
\partial_{t}(\rho\mathbf{u})+ div (\rho\mathbf{u}\otimes\mathbf{u})-\Delta \mathbf{u}+\nabla \Pi=\mathbf{0},\\
 div \mathbf{u}=0,\\
(\rho,\mathbf{u})|_{t=0}=(\rho_{0},\mathbf{u_{0}}).
\end{array}
\right.
\end{equation}
where $\rho$, $\mathbf{u}=(u^{1}, u^{2}, u^{3})$ and $\Pi$ stand for the density, the velocity of the fluid and the pressure, respectively.

The global weak solution to the above system was constructed by Simon \cite{Simon} (See also Lions \cite{Lions}). However, the problem of uniqueness has not been solved. Regularity of such weak solution in three dimension becomes one of the open problems in the mathematical fluid mechanics.

In the case of the smooth initial data without vacuum, Lady\u{z}enskaja and Solonnikov \cite{Ladyzenskaja} addressed the question of the unique solvability of the initial-boundary value problem for the system (\ref{1.1}) in the bounded domain, and Dachin \cite{Dachin1,Dachin2} established the well-posedness of the system (\ref{1.1}) in the
whole space $\R^{d}$. Also, there are some recent progresses \cite{Hammadi2012,Paicu} along this line.

On the other hand, we recall that except the initial data have some special structure, it is
still not known whether or not the system (\ref{1.1}) has a unique global smooth solution with large
smooth initial data, even for the classical Navier-Stokes system, which corresponds to
$\rho = 1$ in (\ref{1.1}). For instance, the global well-posedness result for the classical axisymmetric Navier-Stokes system was firstly proved under no swirl assumption, independently by Ukhovskii and Yudovich \cite{Ukhovskii}, and Ladyzhenskaya \cite{Ladyzhenskaya}, also \cite{Leonardi} for a refined proof. And we \cite{H.Chen} established the global well-posdeness for the classical axisymmetric Navier-Stokes system provided the initial swirl component $u^{\theta}_0$ is sufficient small, i.e.,
 \begin{equation}\label{rho=1}
\|u^{\theta}_0\|_{3}  \leq\frac{1}{ C }\exp\{-  C\|\mathbf{u_{0}}\|_{2}^{2}~(\|\omega_{0}^{\theta}\|_{2}^{2}+ (\|\frac{\omega_{0}^\theta}{r}\|_{2}
+\|\partial_{3}\frac{u_{0}^{\theta}}{r}\|_{2})^{\frac{4}{3}}\|\mathbf{u_{0}}\|_{2}^{2})\},
\end{equation}
 where the right hand side of the above inequality is   scaling invariant.
Recently, H. Abidi, P. Zhang \cite{Hammadi} obtained the global smooth axisymmetric solution without swirl for the inhomogeneous Navier-Stokes equations (\ref{1.1}) when $\|\frac{a_{0}}{r}\|_{\infty}$  is sufficiently small, $a_{0}=\frac{1}{\rho_0}-1$.

Inspired by \cite{Hammadi} and \cite{H.Chen}, we assume that the solution of (\ref{1.1}) is axisymmetric, i.e.,
$$
\rho(t,x)=\rho(t,r,x_{3}), ~\Pi(t,x)=\Pi(t,r,x_{3}),
$$
$$
\mathbf{u}(t,x)=u^{r}(t,r,x_{3})\mathbf{e}_{r}+u^{\theta}(t,r,x_{3})\mathbf{e}_{\theta}+u^{3}(t,r,x_{3})\mathbf{e}_{3},
$$
where
\begin{equation*}
\mathbf{e}_{r}=(\frac{x_{1}}{r},\frac{x_{2}}{r},0),~\mathbf{e}_{\theta}=(-\frac{x_{2}}{r},\frac{x_{1}}{r},0),~\mathbf{e}_{3}=(0,0,1),
\ r=\sqrt{x_{1}^{2}+x_{2}^{2}}.
\end{equation*}
Then from (\ref{1.1}), we have
\begin{equation}\label{1.3}
\left\{\begin{array}{l}
\partial_{t}\rho+\mathbf{u}\cdot \nabla \rho=0,\\
\rho\partial_{t}u^{r}+\rho\mathbf{u}\cdot\nabla u^{r}-(\Delta-\frac{1}{r^{2}})u^{r}-\rho\frac{(u^{\theta})^{2}}{r}+\partial_{r}\Pi=0,\\
\rho\partial_{t}u^{\theta}+\rho\mathbf{u}\cdot\nabla u^{\theta}-(\Delta-\frac{1}{r^{2}})u^{\theta}+\rho\frac{u^{\theta}u^{r}}{r}=0,\\
\rho\partial_{t}u^{3}+\rho\mathbf{u}\cdot\nabla u^{3}-\Delta u^{3}+\partial_{3}\Pi=0,\\
\partial_{r}u^{r}+\frac{1}{r}u^{r}+\partial_{3}u^{3}=0,\\
(u^{r},u^{\theta},u^{3})|_{t=0}=(u_{0}^{r},u_{0}^{\theta},u_{0}^{3}).
\end{array}\right.
\end{equation}

For the axisymmetric velocity field $\mathbf{u}$, we can also compute the vorticity $\boldsymbol{\omega}=\mathrm{curl}~\mathbf{u}$ as follows,
\begin{equation}
\boldsymbol{\omega}=\omega^{r}\mathbf{e}_{r}+\omega^{\theta}\mathbf{e}_{\theta}+\omega^{3}
\mathbf{e}_{3},
\end{equation}
with
\begin{equation}\label{1.4}
\omega^{r}=-\partial_{3}u^{\theta},~\omega^{\theta}=\partial_{3}u^{r}-\partial_{r}u_{3},~\omega^{3}=\partial_{r}u^{\theta}+\frac{u^{\theta}}{r}.
\end{equation}
And we can deduce the equations of vorticity
\begin{equation}\label{1.5}
\left\{\begin{array}{l}
\partial_{t}\omega^{r}+\mathbf{u}\cdot\nabla\omega^{r}+\partial_{3}(\frac{1}{\rho}(\Delta-\frac{1}{r^{2}})u^{\theta})-(\omega^{r}\partial_{r}+\omega^{3}\partial_{3})u^{r}=0,\\
\partial_{t}\omega^{\theta}+\mathbf{u}\cdot\nabla\omega^{\theta}-\partial_{3}(\frac{1}{\rho}((\Delta-\frac{1}{r^{2}})u^{r}-\partial_{r}\Pi))+\partial_{r}(\frac{1}{\rho}(\Delta u^{3}-\partial_{3}\Pi))-\frac{2u^{\theta}\partial_{3}u^{\theta}}{r}-\frac{u^{r}\omega^{\theta}}{r}=0,\\
\partial_{t}\omega^{3}+\mathbf{u}\cdot\nabla\omega^{3}-(\partial_{r}+\frac{1}{r})(\frac{1}{\rho}(\Delta-\frac{1}{r^{2}})u^{\theta})-(\omega^{r}\partial_{r}+\omega^{3}\partial_{3})u^{3}=0,\\
(\omega^{r},\omega^{\theta},\omega^{3})|_{t=0}=(\omega_{0}^{r},\omega_{0}^{\theta},\omega_{0}^{3}).
\end{array}\right.
\end{equation}
Then we state our main theorem, where we set $(\Phi,\Gamma)=(\frac{\omega^{r}}{r},\frac{\omega^{\theta}}{r})$, $\sigma(t)=\min\{t,1\}$, $\langle t\rangle=\sqrt{1+t^{2}}$.
\begin{thm}\label{thm}
Assume  $(\rho_0,\mathbf{u}_0)$ is axisymmetric, $a_{0}=\frac{1}{\rho_0}-1\in L^{2}\cap L^{\infty}$ with $\frac{a_{0}}{r} \in L^{\infty}$, $\mathbf{u}_{0}\in H^{1}$, $ \Gamma_{0},\Phi_{0}\in L^{2} $,   $0<m\leq\rho_{0}\leq M$ with some positive constants $m$ and $M$.
Then there exists a positive time $T_{*}$ so that the system (\ref{1.1}) has a unique solution
$(\rho,\mathbf{u})$ on $[0,T_{*})$, satisfying for any $T<T_{*}$
\begin{equation}\label{1.7}
\begin{array}{l}
\rho \in L^{\infty}(0,T;\R^{3}),~ \mathbf{u}\in \mathcal{C}([0,T];H^{1}(\R^{3})),~\text{and} ~\nabla \mathbf{u}\in L^{2}(0,T;H^{1}(\R^{3})),\\
\underset{t\in[0,T]}{\sup}\left(\sigma(t)(\|\mathbf{u}_{t}(t)\|_{2}^{2}+\|\mathbf{u}(t)\|_{\dot{H}^{2}}^{2}+\|\nabla \Pi(t)\|_{2}^{2})+\int_{0}^{t}\sigma(\tau) \|\nabla\mathbf{u}_{t}(\tau)\|_{2}^{2}~d\tau \right)<\infty.
\end{array}
\end{equation}

In addition,  there exists a positive constant $C=C(m,M)$, such that if
\begin{equation}\label{1.8}
\|u_{0}^{\theta}\|_{3}+\|\frac{a_{0}}{r}\|_{\infty}\|\mathbf{u}_{0}\|_{2}^{2}\leq \eta_{1},\ \
\|\frac{a_{0}}{r}\|_{\infty}^{2}(\|(u_{0}^{\theta})^{2}\|_{2}^{2}+\|\nabla \mathbf{u}_{0}\|_{2}^{2})
\leq \eta_{1}(\|\Gamma_{0}\|_{2}^{2}+\|\Phi_{0}\|_{2}^{2}),
\end{equation}
where
\begin{equation}\label{1.9}
\eta_{1}=\frac{1}{2C}\exp\left(-C\|\mathbf{u}_{0}\|_{2}^{3}(\|\Gamma_{0}\|_{2}+\|\Phi_{0}\|_{2})\right),
\end{equation}
then the solution $(\rho,\mathbf{u})$ is global, i.e. $T_{*}=\infty$. Furthermore, assume that $\mathbf{u}_{0}\in L^{1}$ and $r u_{0}^{\theta}\in L^{1}\cap L^{2}$, we have
\begin{equation}
  \|\mathbf{b}(t)\|_{L^2}^2+\langle t\rangle\|\nabla\mathbf{b}(t)\|_{L^2}^2+t\langle t\rangle\|( \mathbf{b}_t,\Delta\mathbf{b })(t)\|_{L^2}^2\leq
  C \langle t\rangle^{-\frac{3}{2}},\label{1.10-0}
\end{equation}
\begin{equation} \label{1.10}
\begin{array}{l}
\|ru^{\theta}(t)\|_{2}^{2}\leq C \langle t\rangle^{-\frac{3}{2}}, \\
\|u^{\theta}(t)\|_{2}^{2}+\langle t\rangle \|\nabla (u^{\theta}\mathbf{e}_{\theta})(t)\|_{2}^{2}+t\langle t\rangle(\|u_{t}^{\theta}(t)\|_{2}^{2}+\|\Delta(u^{\theta}\mathbf{e}_{\theta})(t)\|_{2}^{2}) \leq C \langle t\rangle^{-\frac{5}{2}}.
\end{array}
\end{equation}
\end{thm}
\begin{rem}
  From the above theorem, we can obtain the global existence of the smooth axisymmetric solution for the  3D inhomogeneous incompressible Navier-Stokes system when
   $$\|\frac{a_{0}}{r}\|_{\infty} \textrm{ and } \|u_{0}^{\theta}\|_{3}  \textrm{ are sufficiently small}.
   $$
It is well-known that the solutions of equations (\ref{1.1}) have scaling properties, as
$$u_{\lambda}(t,x)=\lambda u(\lambda^{2}t,\lambda x),\ \Pi_{\lambda}(t,x)= \lambda^{2}\Pi(\lambda^{2}t,\lambda x),
\ \rho_\lambda(t,x)=\rho(\lambda^{2}t,\lambda x).
$$
We attempt to obtain the global well-posedness result mostly under some scaling invariant conditions. Fortunately,  the inequalities  (\ref{1.8}) are  indeed scaling invariant.
If we choose $\rho_0=1$ in Theorem \ref{thm}, we can obtain the global well-posedness for the 3D classical axisymmetric Navier-Stokes system when
    \begin{equation}
\|u_{0}^{\theta}\|_{3}\leq \frac{1}{2C}\exp\left(-C\|\mathbf{u}_{0}\|_{2}^{3}(\|\Gamma_{0}\|_{2}+\|\Phi_{0}\|_{2})\right).
\end{equation}
The above small condition is better than (\ref{rho=1}).
If we choose $u_0^\theta=0$ in Theorem \ref{thm}, we can obtain the global well-posedness for the 3D inhomogeneous axisymmetric Navier-Stokes system without swirl when
\begin{equation}
\|\frac{a_{0}}{r}\|_{\infty}^{2}(\|\nabla \mathbf{u}_{0}\|_{2}^{2}+\|\mathbf{u}_{0}\|_{2}^{4}\|\Gamma_{0}\|_{2}^{2})
\leq \frac{1}{2C}\exp\left(-C\|\mathbf{u}_{0}\|_{2}^{3}\|\Gamma_{0}\|_{2}\right)\|\Gamma_{0}\|_{2}^{2}.
\end{equation}
This small condition is much clear than that in \cite{Hammadi}.
\end{rem}

\begin{rem}
In \cite{Hammadi} (Section 3), H. Abidi and P. Zhang obtain the following decay estimates,
\begin{equation}\label{decayA}
 \|\mathbf{u}(t)\|_{L^2}^2+\langle t\rangle\|\nabla\mathbf{u}(t)\|_{L^2}^2+t\langle t\rangle\|( \mathbf{u}_t,\Delta\mathbf{u})(t)\|_{L^2}^2\leq
  C \langle t\rangle^{-\frac{3}{2}}.
\end{equation}
The decay estimates  (\ref{decayA}) also  hold for the non-axisymmetric case. One cannot obtain any special behavior for the axisymmetric case from (\ref{decayA}).
In Theorem \ref{thm}, we obtain that the swirl component $u^{\theta}$ will share better decay estimates than $(u^{r}, u^{3})$. One can easily show that these decay estimates are optimal under the conditions $\rho_0\equiv1$, $\mathbf{u}_{0}\in L^{1}\cap H^2$ and $r u_{0}^{\theta}\in L^{1}\cap L^{2}$.
\end{rem}

Thanks to the blow up criteria (for example, see \cite{Kim}), to prove the global well-posedness, we only need to prove that $\|\nabla \mathbf{u}\|_{L^{\infty,2}_T}$ is bounded for all $T>0$.
For the axisymmetric solution of (\ref{1.1}) without swirl, for example, the homogeneous case \cite{Leonardi,Ladyzhenskaya,Ukhovskii} or the inhomogeneous case with $\|\frac{a_{0}}{r}\|_{\infty}$ sufficiently small \cite{Hammadi}, the authors are used to prove
$$\|\Gamma(t)\|_{2}^{2} \leq constant,~\forall t\in(0,\infty),$$
then
    $$
    \|\nabla \mathbf{u}\|_2\approx\|w^\theta\|_{2}\leq constant,~\forall t\in(0,\infty).
    $$
However, when the solutions have nonzero swirls, the estimate of $\|\Gamma(t)\|_{2}$ will depend on many complicated terms. For the homogeneous case,
 we  \cite{H.Chen} find that the system of the pair $(\Phi,\Gamma)$ has some good structures, and we easily  show that
$$\|\Gamma(t)\|_{2}^{2}+\|\Phi(t)\|_{2}^{2} \leq constant,~\forall t\in(0,\infty).
$$
In such sense of consideration,  we consider the following system for the pair $(\Phi,\Gamma)$,
\begin{equation}\label{1.6}
\left\{
\begin{array}{l}
\partial_{t}\Phi+\mathbf{u}\cdot\nabla\Phi+\frac{1}{r}\partial_{3}(\frac{1}{\rho}(\Delta-\frac{1}{r^{2}})u^{\theta})-(\omega^{r}\partial_{r}+\omega^{3}\partial_{3})\frac{u^{r}}{r}=0,\\
\partial_{t}\Gamma+\mathbf{u}\cdot\nabla\Gamma-\frac{1}{r}\partial_{3}(\frac{1}{\rho}((\Delta-\frac{1}{r^{2}})u^{r}-\partial_{r}\Pi))+\frac{1}{r}\partial_{r}(\frac{1}{\rho}(\Delta u^{3}-\partial_{3}\Pi))+2\frac{u^{\theta}}{r}\Phi=0.
\end{array}
\right.
\end{equation}
If we assume $a=\frac{1}{\rho}-1, a|_{r=0}=0$, we also have the following important new identity,
\begin{eqnarray*}
&&\frac{1}{2}\frac{d}{dt}(\|\Phi\|_{2}^{2}+\|\Gamma\|_{2}^{2})+\|\nabla\Phi\|_{2}^{2}+\|\nabla\Gamma\|_{2}^{2}\\
&=& \int_{\R^{3}}[(\omega^{r}\partial_{r}+\omega^{3}\partial_{3})\frac{u^{r}}{r}~\Phi-2\Gamma\Phi \\ &&+\frac{a}{r}(\partial_{r}\omega^{3}-\partial_{3}\omega^{r})\partial_{3}\Phi-\frac{a}{r}(\partial_{3}
\omega^{\theta}-\partial_{r}\Pi)\partial_{3}\Gamma-\frac{a}{r}(\partial_{r}\omega^{\theta}+\Gamma+\partial_{3}\Pi)\partial_{r}\Gamma]~dx.
\end{eqnarray*}
This is the key ingredient for us to obtain some \textit{a priori} estimates for the inhomogeneous axisymmetric Navier-Stokes system  (\ref{1.3}). However, this identity contains  many more complicated terms compared with \cite{H.Chen}. Fortunately, it can be controlled  along by the estimates (\ref{priori}) and (\ref{2.17}). Then we can reach the goal by the continuous method under the small assumptions (\ref{1.8}). We may need to point out 
 there are  two technical steps in our proofs:
\begin{description}
  \item[(1)] using $\|\Gamma(t)\|_{2}$ and the energy method to estimate $\|\boldsymbol{\omega}\|_2$ (see Lemmas \ref{l2.6-0} and \ref{l2.7});
  \item[(2)]  using the energy method to estimate $\|\Gamma(t)\|_{2}+\|\Phi(t)\|_{2}$ (see Lemma \ref{l2.6}).
\end{description}
Furthermore,
since there is no pressure term in the equations of $(w^r,w^3)$ (\ref{1.5}), one can use the similar argument as that in the homogeneous case \cite{P.Zhang}, using $\|\mathbf{b}\|_{L^\infty}$ to     estimate $\|w^r\|_2+\|w^3\|_2$. But, in our case, we have to give a new estimate for  $\|w^r\|_2+\|w^3\|_2$ in Lemma \ref{l2.6-0}.

\noindent\textbf{Notations.} We denote $\tilde\nabla=(\partial_{r},\partial_{3}), \mathbf{\tilde u}=(u^{r},u^{3}), \mathbf{b}=u^{r}\mathbf{e}_{r}+u^{3}\mathbf{e}_{3},$ and if $f(x)$ is axisymmetric, i.e. $f(x)=f(r,x_{3})$, we have
$$
\mathbf{u}\cdot\nabla f=\mathbf{b}\cdot\nabla f= (u^{r}\partial_{r}+u^{3}\partial_{3})f.
$$
We introduce the Banach space $L^{p,q}_{T}$, equipped with norm
$$
\|f\|_{L^{p,q}_{T}}=\left\{
\begin{array}{lcl}
\bigl(\int_{0}^{T}~\|f(t)\|^{p}_{q}~dt\bigr)^{\frac{1}{p}},&~ & \text{if}~1\leq p<\infty,\\
\text{ess sup}_{t\in(0,T)}\|f(t)\|_{q}, &~ & \text{if}~p=\infty,
\end{array}\right.
$$
where
$$
\|f(t)\|_{q}=\left\{\begin{array}{lcl}\bigl(\int_{\R^{3}}~|f(t,x)|^{q}~dx \bigr)^{\frac{1}{q}},&~ & \text{if}~1\leq q<\infty,\\
\text{ess sup}_{x\in\R^{3}}~|f(t,x)|, &~ & \text{if}~q=\infty.
\end{array}\right.
$$

\section{Preliminaries}
From Lemmas 2.2-2.4 in \cite{H.Chen}, we present the following proposition of the axisymmetric velocity, which is frequently used in the axisymmetric system.
\begin{prop} \label{prop}
 Assume $(\rho,\mathbf{u})$ is the smooth axisymmetric solution of (\ref{1.1}) on $[0,T]$, with the initial data $\mathbf{u_{0}}$, and $\mathrm{curl} ~\mathbf{u}= \boldsymbol{\omega} $, then
\begin{enumerate}

\item[\textrm{i})] $\mathbf{u}=u^{\theta}\mathbf{e}_{\theta}+\nabla\times(\psi \mathbf{e}_{\theta})=-\partial_{3}\psi\boldsymbol{e_{r}}+u^{\theta}\mathbf{e}_{\theta}+\frac{\partial_{r}(r\psi)}{r}\boldsymbol{e_{3}}$, with
\begin{equation*}
u^{\theta}(t,r,x_{3}), \ \
\psi(t,r,x_{3}), \ \
\omega^{\theta}(t,r,x_{3})\in C^{1}(0,T; C^{\infty}(\overline{\R^{+}}\times\R)),
\end{equation*}
and $u^{\theta}(t,0,x_{3})=\psi(t,0,x_{3})=\omega^{\theta}(t,0,x_{3})=0$.

\item[\textrm{ii})] There exists a  constant $C=C(q)$, such that for $\forall t\in [0,T]$ and $1<q<\infty$,
\begin{equation}\label{2.1}
\|\tilde\nabla u^{r}\|_{q}+\|\tilde\nabla u^{3}\|_{q}+\|\frac{u^{r}}{r}\|_{q}\leq C \|\omega^{\theta}\|_{q},
\end{equation}
    $$
\|\tilde\nabla u^{\theta}\|_{q}+\|\frac{u^{\theta}}{r}\|_{q}\leq C\|\nabla \mathbf{u}\|_{q}.
$$
\item[\textrm{iii})] \begin{equation*}
\frac{u^{r}}{r}=\Delta^{-1}\partial_{3}(\Gamma)-2\frac{\partial_{r}}{r}\Delta^{-2}\partial_{3}(\Gamma).
\end{equation*}
There exists a  constant $C=C(q)$, such that for $1<q<\infty$,
\begin{equation}\label{2.2}
\| \tilde\nabla\frac{u^{r}}{r}\|_{q}\leq C(q)~\|\Gamma\|_{q},
\end{equation}
\begin{equation*}
\|\tilde\nabla\tilde\nabla \frac{u^{r}}{r}\|_{q}\leq C(q)~\|\partial_{3} (\Gamma)\|_{q},
\end{equation*}
and
\begin{equation}\label{infty}
\| \frac{u^{r}}{r}\|_{\infty}\leq C~\|\Gamma\|_{2}^{\frac{1}{2}} \|\nabla\Gamma\|_{2}^{\frac{1}{2}}.
\end{equation}
\item[\textrm{iv})] Sobolev-Hardy inequality. If $0\leq s<2,q_{*}\in[2,2(3-s)]$, then there exists a positive constant $C_{q_*,s}$, such that for all $f\in C_{0}^{\infty}(\R^{3})$,
\begin{equation*}
 \left\|\frac{f}{r^{\frac{s}{q_{*}}}}\right\|_{q_{*}} \leq C_{q_*,s}\|f\|_2^{\frac{3-s}{q_*}-\frac{1}{2}}\|\nabla f\|_2^{\frac{3}{2}-\frac{3-s}{q_*}}.
 \end{equation*}
\end{enumerate}
\end{prop}

We can extend the properties in \cite{Hammadi} to the axisymmetric velocity with nonzero swirls, and have following identities.
\begin{lem}\label{lem1}
Under the conditions in Proposition \ref{prop}, we have
\begin{equation}
(\Delta-\frac{1}{r^{2}})u^{r}=\partial_{3}\omega^{\theta},
\end{equation}
\begin{equation}
(\Delta-\frac{1}{r^{2}})u^{\theta}=\partial_{r}\omega^{3}-\partial_{3}\omega^{r},
\end{equation}
\begin{equation}
\Delta u^{3}=-\partial_{r}\omega^{\theta}-\Gamma,\ \
\nabla \cdot \boldsymbol{\omega}=0.
\end{equation}
\end{lem}
\begin{proof} They can be deduced directly from (\ref{1.4}) and the divergence-free property of $\mathbf{u}$. For instance,
\begin{eqnarray*}
(\Delta-\frac{1}{r^{2}})u^{\theta}&=&(\partial_{r}^{2}+\partial_{3}^{2}+\frac{1}{r}\partial_{r}-\frac{1}{r^{2}})u^{\theta} \\
&=&\partial_{r}(\omega^{3}-\frac{u^{\theta}}{r})-\partial_{3}\omega^{r}+(\frac{1}{r}\partial_{r}-\frac{1}{r^{2}})u^{\theta}\\
&=&\partial_{r}\omega^{3}-\partial_{3}\omega^{r}.
\end{eqnarray*}
Using the similar argument as that in \cite{Hammadi}, we have
\begin{eqnarray*}
(\Delta-\frac{1}{r^{2}})u^{r}&=&(\partial_{r}^{2}+\partial_{3}^{2}+\frac{1}{r}\partial_{r}-\frac{1}{r^{2}})u^{r}\\
&=&\partial_{r}(-\frac{u^{r}}{r}-\partial_{3}u^{3})+\partial_{3}(\omega^{\theta}+\partial_{r}u^{3})+(\frac{1}{r}\partial_{r}-\frac{1}{r^{2}})u^{r}\\
&=&\partial_{3}\omega^{\theta},
\end{eqnarray*}
and
 $$\Delta u^{3}=-\partial_{r}\omega^{\theta}-\Gamma.$$
\end{proof}

Thus we can propose the following remark  which is essential in the proof of Theorem \ref{thm}.
\begin{rem}
Set $\mathbf{B}=\omega^{r}\mathbf{e}_{r}+\omega^{3}\mathbf{e}_{3}$, and
$$
\nabla \cdot \mathbf{B}=0,~\ \nabla \times \mathbf{B}=(\partial_{3}\omega^{r}-\partial_{r}\omega^{3})\mathbf{e_{\theta}} .
$$
Thus
\begin{equation}\label{27}
\|\nabla\omega^{r},\nabla\omega^{3}\|_{p}+\|\Phi\|_{p}\leq C\|\nabla \mathbf{B}\|_{p}\leq C \|\partial_{3}\omega^{r}-\partial_{r}\omega^{3}\|_{p}, 1<p<\infty.
\end{equation}
\end{rem}
\subsection{A priori estimates}

Now we shall present some  \textit{a priori} estimates in this section.

One can easily obtain the following lemma and omit the detail, see \cite{Lions}.
\begin{lem}{\label{l2.3}}
Under the conditions in Proposition \ref{prop}, we obtain that for all $t\in [0,T]$,
\begin{equation}
0<m\leq \rho\leq M,\label{den}
\end{equation}
and the energy inequality,
\begin{equation}\label{2.3}
\frac{1}{2}\|\sqrt{\rho}\mathbf{u}\|_{2}^{2}+\int_{0}^{t}\|\nabla\mathbf{u}\|_{2}^{2}\leq C \|\mathbf{u_{0}}\|_{2}^{2}.
\end{equation}
\end{lem}

For the convenience of the proof, we estimate the swirl component and the convection term below. The proofs of these two
lemmas will be given in the Appendix.
\begin{lem}\label{l2.4}
Under the conditions in Proposition \ref{prop}, we obtain that for all $t\in [0,T]$,
\begin{eqnarray}\label{2.4}
\frac{d}{dt}\left\|\sqrt{\rho}(u^{\theta})^{2}\right\|_{2}^{2}+\left\|\nabla (u^{\theta})^{2}\right\|_{2}^{2}
+\left\|\frac{(u^{\theta})^{2}}{r}\right\|_{2}^{2}
\leq C\|\mathbf{u}\|_{2}^{\frac{4}{3}}\|\Gamma\|_{2}^{\frac{4}{3}}\|\nabla \mathbf{u}\|_{2}^{\frac{10}{3}}.
\end{eqnarray}
\end{lem}

\begin{lem}\label{lem2.5}
Under the conditions in Proposition \ref{prop}, we obtain that for all $t\in [0,T]$,
\begin{eqnarray}\label{cross}
\|\mathbf{u}\cdot\nabla\mathbf{\tilde u}\|_{2}^{2}&=&\|\mathbf{u}\cdot\nabla u^{r}\|_{2}^{2}+\|\mathbf{u}\cdot\nabla u^{3}\|_{2}^{2}\nonumber\\
&\leq& C_{\delta}\|\mathbf{u}\|_{2}^{\frac{4}{3}}\|\Gamma\|_{2}^{\frac{4}{3}}\|\nabla \mathbf{u}\|_{2}^{\frac{10}{3}}
+\delta(\|\nabla\omega^{\theta}\|_{2}^{2}+\|\Gamma\|_{2}^{2}),
\end{eqnarray}
where $\delta$ is sufficient small.
\end{lem}

We now evaluate the terms $\|\omega^{r}\|_{2}$ and $\|\omega^{3}\|_{2}$ by the system (\ref{1.5}).
\begin{lem}\label{l2.6-0}
Under the conditions in Proposition \ref{prop}, we obtain that for all $t\in [0,T]$,
\begin{eqnarray}\label{priori2}
&&\frac{d}{dt}(\|\omega^{r}\|_{2}^{2}+\|\omega^{3}\|_{2}^{2})+\|\tilde\nabla\omega^{r}\|_{2}^{2}+\|\tilde\nabla\omega^{3}\|_{2}^{2}
+\|\Phi\|_{2}^{2}\nonumber\\
&\leq& C_{\delta}\|\mathbf{u}\|_{2}^{\frac{4}{3}}\|\Gamma\|_{2}^{\frac{4}{3}}\|\nabla \mathbf{u}\|_{2}^{\frac{10}{3}}
+\delta\|\nabla\boldsymbol{\omega}\|_{2}^{2},
\end{eqnarray}
where  $\delta$ is sufficiently small.
\end{lem}
\begin{proof}
Multiplying the equations $(\ref{1.5})_{1}$ and $(\ref{1.5})_{3}$ by $\omega^{r}$ and $\omega^{3}$, respectively,  using integration by parts and Lemma \ref{lem1}, we have
\begin{eqnarray}\label{2.15}
&&\frac{1}{2}\frac{d}{dt}(\|\omega^{r}\|_{2}^{2}+\|\omega^{3}\|_{2}^{2})+\int_{\R^{3}}\frac{1}{\rho}(\partial_{3}\omega^{r}-\partial_{r}\omega^{3})^{2}~dx\nonumber\\
&=&\int_{\R^{3}}~(\omega^{r}\partial_{r}+\omega^{3}\partial_{3})u^{r}\omega^{r}+(\omega^{r}\partial_{r}+\omega^{3}\partial_{3})u^{3}\omega^{3}~dx\nonumber \\
&=&J_{1}+J_{2}+J_{3}+J_{4}.
\end{eqnarray}

Take $v_{1}=(u^{r},u^{\theta}),v_{2}=(\omega^{r},\omega^{\theta},\omega^{3})$ and $r_{0}$ as in (\ref{r0}). By the similar calculus as that in the proof of Lemma \ref{lem2.5}, we have
\begin{eqnarray}\label{cross2}
\|v_{1}v_{2}|_{r>r_{0}}\|_{2}^{2}&\leq&(\int_{\R}\int_{r_{0}}^{\infty}|r^{\frac{1}{2}}v_{1}|^{4}drdx_{3})^{\frac{1}{2}}(\int_{\R}\int_{r_{0}}^{\infty}|v_{2}|^{4}drdx_{3})^{\frac{1}{2}}\nonumber\\
&\leq& C (\int_{\R}\int_{r_{0}}^{\infty}|r^{\frac{1}{2}}v_{1}|^{2}drdx_{3})^{\frac{1}{2}}(\int_{\R}\int_{r_{0}}^{\infty}|\tilde\nabla (r^{\frac{1}{2}}v_{1})|^{2}drdx_{3})^{\frac{1}{2}}\nonumber\\
&&\times(\int_{\R}\int_{r_{0}}^{\infty}|v_{2}|^{2} drdx_{3})^{\frac{1}{2}}(\int_{\R}\int_{r_{0}}^{\infty}|\tilde\nabla v_{2}|^{2} drdx_{3})^{\frac{1}{2}}\nonumber\\
&\leq& \frac{C}{r_{0}} \|\mathbf{u}\|_{2} \|\nabla \mathbf{u}\|_{2}^{2}\|\nabla\boldsymbol{\omega}\|_{2}\nonumber\\
&\leq&C\|\mathbf{u}\|_{2}^{\frac{2}{3}}\|\Gamma\|_{2}^{\frac{2}{3}}\|\nabla \mathbf{u}\|_{2}^{\frac{5}{3}}\|\nabla\boldsymbol{\omega}\|_{2},
\end{eqnarray}
and
\begin{eqnarray}\label{cross3}
\|u^{r}v_{2}|_{r\leq r_{0}}\|_{2}^{2} &\leq& r_{0}^{2}\|\frac{u^{r}}{r}v_{2}\|_{2}^{2}\nonumber\\
&\leq & Cr_{0}^{2}\|\nabla\frac{u^{r}}{r}\|_{2}^{2}\|v_{2}\|_{3}^{2}\nonumber\\
&\leq & Cr_{0}^{2}\|\Gamma\|_{2}^{2}\|\nabla\mathbf{u}\|_{2}\|\nabla\boldsymbol{\omega}\|_{2}.\nonumber\\
&\leq&C\|\mathbf{u}\|_{2}^{\frac{2}{3}}\|\Gamma\|_{2}^{\frac{2}{3}}\|\nabla \mathbf{u}\|_{2}^{\frac{5}{3}}\|\nabla\boldsymbol{\omega}\|_{2}.
\end{eqnarray}
Using integration by parts,   Cauchy inequality, (\ref{2.1}), (\ref{2.2}), (\ref{cross2}), (\ref{cross3}) and the fact (\ref{1.4}), we have
\begin{eqnarray*}
J_{1}&=&\int_{\R^{3}}\partial_{r}u^{r}\omega^{r}\omega^{r}~dx\\
&=&\int_{\R^{3}}-2 u^{r}\omega^{r}\partial_{r}\omega^{r}-u^{r}\omega^{r}\Phi~dx\\
&\leq&C\|u^{r}\omega^{r}\|_{2}\|\nabla\boldsymbol{\omega}\|_{2}\\
&\leq&C(\|u^{r}\omega^{r}|_{r\leq r_{0}}\|_{2}+\|u^{r}\omega^{r}|_{r> r_{0}}\|_{2})\|\nabla\boldsymbol{\omega}\|_{2}\\
&\leq&C\|\mathbf{u}\|_{2}^{\frac{1}{3}}\|\Gamma\|_{2}^{\frac{1}{3}}\|\nabla \mathbf{u}\|_{2}^{\frac{5}{6}}\|\nabla\boldsymbol{\omega}\|_{2}^{\frac{3}{2}},
\end{eqnarray*}
and
\begin{eqnarray*}
J_{2}&=&\int_{\R^{3}}-u^{r}\partial_{3}(\omega^{r}\omega^{3})~dx\\
&\leq&C\|u^{r}v_{2}\|_{2}\|\nabla\boldsymbol{\omega}\|_{2}\\
&\leq&C(\|u^{r}v_{2}|_{r\leq r_{0}}\|_{2}+\|u^{r}v_{2}|_{r> r_{0}}\|_{2})\|\nabla\boldsymbol{\omega}\|_{2}\\
&\leq&C\|\mathbf{u}\|_{2}^{\frac{1}{3}}\|\Gamma\|_{2}^{\frac{1}{3}}\|\nabla \mathbf{u}\|_{2}^{\frac{5}{6}}\|\nabla\boldsymbol{\omega}\|_{2}^{\frac{3}{2}}.
\end{eqnarray*}
Similarly, using $\partial_{r}u^{3}=\partial_{3}u^{r}-\omega^{\theta}$, we have
\begin{eqnarray*}
J_{3}&=&J_{2}-\int_{\R^{3}}\omega^{\theta}\omega^{r}\omega^{3}~dx\\
&=&J_{2}-\int_{\R}\int_{r\leq r_{0}}r\Gamma\omega^{r}\omega^{3} ~rdrdx_{3}+\int_{\R}\int_{r>r_{0}}\partial_{3}u^{\theta}\omega^{\theta}\omega^{3}rdrdx_{3}\\
&\leq&|J_{2}|+r_{0}\int_{r\leq r_{0}}|\Gamma\omega^{r}\omega^{3}|~dx++\int_{\R}\int_{r>r_{0}}|u^{\theta}||\partial_{3}(\omega^{\theta}\omega^{3})|~dx\\
&\leq&|J_{2}|+r_{0}\|\Gamma\|_{2}\|\nabla \mathbf{u}\|_{2}^{\frac{1}{2}}\|\nabla\boldsymbol{\omega}\|_{2}^{\frac{3}{2}}+\|u^{\theta}v_{2}|_{r>r_{0}}\|_{2}\|\nabla\boldsymbol{\omega}\|_{2}\\
&\leq&|J_{2}|+C\|\mathbf{u}\|_{2}^{\frac{1}{3}}\|\Gamma\|_{2}^{\frac{1}{3}}\|\nabla \mathbf{u}\|_{2}^{\frac{5}{6}}\|\nabla\boldsymbol{\omega}\|_{2}^{\frac{3}{2}}\\
&\leq&C\|\mathbf{u}\|_{2}^{\frac{1}{3}}\|\Gamma\|_{2}^{\frac{1}{3}}\|\nabla \mathbf{u}\|_{2}^{\frac{5}{6}}\|\nabla\boldsymbol{\omega}\|_{2}^{\frac{3}{2}},
\end{eqnarray*}
and
\begin{eqnarray*}
J_{4}&=&\int_{\R^{3}}-\frac{\partial_{r}(ru^{r})}{r}(\omega^{3})^{2}~dx\\
&=&\int_{\R^{3}}2u^{r}\omega^{3}\partial_{r}\omega^{3}~dx\\
&\leq&C\|\mathbf{u}\|_{2}^{\frac{1}{3}}\|\Gamma\|_{2}^{\frac{1}{3}}\|\nabla \mathbf{u}\|_{2}^{\frac{5}{6}}\|\nabla\boldsymbol{\omega}\|_{2}^{\frac{3}{2}}.
\end{eqnarray*}
Combining the above inequalities, we have
\begin{eqnarray*}
J_{1}+J_{2}+J_{3}+J_{4} &\leq&C\|\mathbf{u}\|_{2}^{\frac{1}{3}}\|\Gamma\|_{2}^{\frac{1}{3}}\|\nabla \mathbf{u}\|_{2}^{\frac{5}{6}}\|\nabla\boldsymbol{\omega}\|_{2}^{\frac{3}{2}}\\
&\leq& C_{\delta}\|\mathbf{u}\|_{2}^{\frac{4}{3}}\|\Gamma\|_{2}^{\frac{4}{3}}\|\nabla \mathbf{u}\|_{2}^{\frac{10}{3}}
+\delta\|\nabla\boldsymbol{\omega}\|_{2}^{2}.
\end{eqnarray*}
Recall that the density has lower bound $ \rho \geq m>0$ and (\ref{27}), we have (\ref{priori2}).
\end{proof}

We present an essential estimate of $\|\nabla\mathbf{u}\|_{L_{T}^{\infty,2}}$ as follows.
\begin{lem}\label{l2.7}
Under the conditions in Proposition \ref{prop}, we obtain that for all $t\in [0,T]$,
\begin{eqnarray}\label{priori}
&&\left\|(u^{\theta})^{2}\right\|_{L_{t}^{\infty,2}}^{2}+\|\boldsymbol \omega\|_{L_{t}^{\infty,2}}^{2}\nonumber\\
&&+\left\|\nabla (u^{\theta})^{2}\right\|_{L_{t}^{2,2}}^{2}
+\left\|\frac{(u^{\theta})^{2}}{r}\right\|_{L_{t}^{2,2}}^{2}+\|u_{t}^{r}\|_{L_{t}^{2,2}}^{2} +\|u_{t}^{3}\|_{L_{t}^{2,2}}^{2}+\|\nabla \Pi\|_{L_{t}^{2,2}}^{2}+\|\nabla \boldsymbol\omega\|_{L_{t}^{2,2}}^{2}\nonumber\\
&\leq&C(\|(u_{0}^{\theta})^{2}\|_{2}^{2}+\|\nabla \mathbf{u}_{0}\|_{2}^{2}+\|\mathbf{u}_{0}\|_{2}^{4}\|\Gamma\|_{L_{t}^{\infty,2}}^{2}) \exp( C\|\mathbf{u}_{0}\|_{2}^{3}\|\Gamma\|_{L_{t}^{\infty,2}}).
\end{eqnarray}
\end{lem}
\begin{proof}
$\bullet$  The $\dot{H}^{1}$ estimates of $u^{r},u^{3}$.

Multiplying the equations $(\ref{1.3})_{2}$ and $(\ref{1.3})_{4} $ by $\partial_{t}u^{r}$ and $\partial_{t}u^{3}$, respectively, using integration by parts and Cauchy inequality, we have
\begin{eqnarray}\label{2.5}
&&\frac{1}{2}\frac{d}{dt}(\|\nabla u^{r}\|_{2}^{2}+\|\nabla u^{3}\|_{2}^{2}+\|\frac{u^{r}}{r}\|_{2}^{2})+\|\sqrt{\rho}u_{t}^{r}\|_{2}^{2} +\|\sqrt{\rho}u_{t}^{3}\|_{2}^{2}\nonumber \\
&=& -\left(\int_{\R^{3}}~\rho\mathbf{u}\cdot\nabla u^{r}u_{t}^{r}-\rho\frac{(u^{\theta})^{2}}{r}u_{t}^{r}+\rho\mathbf{u}\cdot\nabla u^{3}u_{t}^{3}~dx\right)
\nonumber\\
&\leq& C(\|\mathbf{u}\cdot\nabla u^{r}\|_{2}^{2}+\|\mathbf{u}\cdot\nabla u^{3}\|_{2}^{2}+\|\frac{(u^{\theta})^{2}}{r}\|_{2}^{2})+\frac{1}{2}(\|\sqrt{\rho}u_{t}^{r}\|_{2}^{2} +\|\sqrt{\rho}u_{t}^{3}\|_{2}^{2}) .
\end{eqnarray}

$\bullet$  The estimates of  $\Pi$ and $\nabla \omega^{\theta}$ by the Stokes equation.

By Lemma \ref{lem1}, we can deduce the stokes system
$$\left\{
\begin{array}{rll}
-\partial_{3}\omega^{\theta}+\partial_{r}\Pi&=&-\rho\partial_{t}u^{r}-\rho\mathbf{u}\cdot\nabla u^{r}+\rho\frac{(u^{\theta})^{2}}{r} \ ,\\
\partial_{r}\omega^{\theta}+\Gamma+\partial_{3}\Pi&=&-\rho\partial_{t}u^{3}-\rho\mathbf{u}\cdot\nabla u^{3} .\\
\end{array}\right.
$$
Multiplying the above equations  by $\partial_{r}\Pi$ and $\partial_{3}\Pi$ respectively, using integration by parts,  Cauchy inequality and the fact that $\omega^{\theta}|_{r=0}=0$, we have
\begin{eqnarray}\label{2.6}
\|\nabla \Pi\|_{2}^{2}&=&-\left(\int_{\R^{3}}~\rho\mathbf{u}\cdot\nabla u^{r}\partial_{r}\Pi-\rho\frac{(u^{\theta})^{2}}{r}\partial_{r}\Pi+\rho\mathbf{u}\cdot\nabla u^{3} \partial_{3}\Pi~dx\right)\\
&&-\left(\int_{\R^{3}}~\rho\partial_{t}u^{r}\partial_{r}\Pi+\rho\partial_{t}u^{3}\partial_{3}\Pi~dx\right)\nonumber\\
&\leq&C(\|\mathbf{u}\cdot\nabla u^{r}\|_{2}^{2}+\|\mathbf{u}\cdot\nabla u^{3}\|_{2}^{2}+\|\frac{(u^{\theta})^{2}}{r}\|_{2}^{2})+C(\|\sqrt{\rho}u_{t}^{r}\|_{2}^{2} +\|\sqrt{\rho}u_{t}^{3}\|_{2}^{2})+\frac{1}{2}\|\nabla\Pi\|_{2}^{2} .\nonumber
\end{eqnarray}

Along the same line,  multiplying the above system by $ -\partial_{3}\omega^{\theta}$ and $\partial_{r}\omega^{\theta}+\Gamma$ respectively, we have
\begin{eqnarray}\label{2.7}
\|\nabla \omega^{\theta}\|_{2}^{2}+\|\Gamma\|_{2}^{2}&=&\int_{\R^{3}} (\rho\partial_{t}u^{r}+\rho\mathbf{u}\cdot\nabla u^{r}-\rho\frac{(u^{\theta})^{2}}{r})\partial_{3}\omega^{\theta} ~dx\nonumber\\
&&-\int_{\R^{3}}(\rho\partial_{t}u^{3}+\rho\mathbf{u}\cdot\nabla u^{3})(\partial_{r}\omega^{\theta}+\Gamma) ~dx \nonumber\\
&\leq& C(\|\mathbf{u}\cdot\nabla u^{r}\|_{2}^{2}+\|\mathbf{u}\cdot\nabla u^{3}\|_{2}^{2}+\|\frac{(u^{\theta})^{2}}{r}\|_{2}^{2})+C(\|\sqrt{\rho}u_{t}^{r}\|_{2}^{2} +\|\sqrt{\rho}u_{t}^{3}\|_{2}^{2})\nonumber\\
&&+\frac{1}{2}(\|\nabla \omega^{\theta}\|_{2}^{2}+\|\Gamma\|_{2}^{2}).
\end{eqnarray}

Combining (\ref{2.6}) and (\ref{2.7}),  we get
\begin{eqnarray}\label{2.8}
&&\|\nabla \Pi\|_{2}^{2}+\|\nabla \omega^{\theta}\|_{2}^{2}+\|\Gamma\|_{2}^{2}\nonumber\\
&\leq& C(\|\mathbf{u}\cdot\nabla u^{r}\|_{2}^{2}+\|\mathbf{u}\cdot\nabla u^{3}\|_{2}^{2}+\|\frac{(u^{\theta})^{2}}{r}\|_{2}^{2})+C(\|\sqrt{\rho}u_{t}^{r}\|_{2}^{2} +\|\sqrt{\rho}u_{t}^{3}\|_{2}^{2}) .
\end{eqnarray}
Combining the above estimates (\ref{2.4}), (\ref{2.5}), and (\ref{2.8}), applying (\ref{cross}), we obtain
\begin{eqnarray}\label{priori1}
&&\frac{d}{dt}(\left\|\sqrt{\rho}(u^{\theta})^{2}\right\|_{2}^{2}+\|\nabla u^{r}\|_{2}^{2}+\|\nabla u^{3}\|_{2}^{2}+\|\frac{u^{r}}{r}\|_{2}^{2})+\left\|\nabla (u^{\theta})^{2}\right\|_{2}^{2}
+\left\|\frac{(u^{\theta})^{2}}{r}\right\|_{2}^{2}\nonumber\\
&&+(\|\sqrt{\rho}u_{t}^{r}\|_{2}^{2} +\|\sqrt{\rho}u_{t}^{3}\|_{2}^{2}+\|\nabla \Pi\|_{2}^{2}+\|\nabla \omega^{\theta}\|_{2}^{2}+\|\Gamma\|_{2}^{2})\nonumber\\
&\leq& C( \|\mathbf{u}\cdot\nabla u^{r}\|_{2}^{2}+\|\mathbf{u}\cdot\nabla u^{3}\|_{2}^{2}+\|\mathbf{u}\|_{2}^{\frac{4}{3}}\|\Gamma\|_{2}^{\frac{4}{3}}\|\nabla \mathbf{u}\|_{2}^{\frac{10}{3}}) \nonumber\\
&\leq&C_{\delta}\|\mathbf{u}\|_{2}^{\frac{4}{3}}\|\Gamma\|_{2}^{\frac{4}{3}}\|\nabla \mathbf{u}\|_{2}^{\frac{10}{3}}+\delta(\|\nabla\omega^{\theta}\|_{2}^{2}+\|\Gamma\|_{2}^{2}).
\end{eqnarray}

Combining the inequalities (\ref{priori2}) and (\ref{priori1}), we obtain
\begin{eqnarray*}
&&\frac{d}{dt}(\left\|\sqrt{\rho}(u^{\theta})^{2}\right\|_{2}^{2}+\|\nabla u^{r}\|_{2}^{2}+\|\nabla u^{3}\|_{2}^{2}+\|\frac{u^{r}}{r}\|_{2}^{2}+\|\omega^{r}\|_{2}^{2}+\|\omega^{3}\|_{2}^{2})+\left\|\nabla (u^{\theta})^{2}\right\|_{2}^{2}
+\left\|\frac{(u^{\theta})^{2}}{r}\right\|_{2}^{2}\\
&&+(\|\sqrt{\rho}u_{t}^{r}\|_{2}^{2} +\|\sqrt{\rho}u_{t}^{3}\|_{2}^{2}+\|\nabla \Pi\|_{2}^{2}+\|\nabla \omega^{\theta}\|_{2}^{2}+\|\Gamma\|_{2}^{2}+\|\nabla\omega^{r}\|_{2}^{2}+\|\nabla\omega^{3}\|_{2}^{2}+\|\Phi\|_{2}^{2})\\
&\leq&C\|\mathbf{u}\|_{2}^{\frac{4}{3}}\|\Gamma\|_{2}^{\frac{4}{3}}\|\nabla \mathbf{u}\|_{2}^{\frac{10}{3}}\\
&\leq&C(\|\mathbf{u}\|_{2}^{2}\|\Gamma\|_{2}^{2}\|\nabla\mathbf{u}\|_{2}^{2})^{\frac{1}{3}}(\|\mathbf{u}\|_{2}\|\Gamma\|_{2}\|\nabla\mathbf{u}\|_{2}^{4})^{\frac{2}{3}}\\
&\leq&C\|\mathbf{u}\|_{2}^{2}\|\Gamma\|_{2}^{2}\|\nabla\mathbf{u}\|_{2}^{2}+C\|\mathbf{u}\|_{2}\|\Gamma\|_{2}\|\nabla\mathbf{u}\|_{2}^{4}.
\end{eqnarray*}
Apply the Grownwall's inequality and Lemma \ref{l2.3}, we have (\ref{priori}).

\end{proof}

Using the ideas in \cite{H.Chen}, we consider the $L^{2}$ estimate of the pair $(\Phi,\Gamma)$ as follows.

\begin{lem}\label{l2.6}
Under the conditions in Proposition \ref{prop}, and
assume $a=1/\rho-1,a|_{r=0}=0$, we obtain that for all $t\in [0,T]$,
\begin{eqnarray}\label{2.14}
&&\frac{1}{2}\frac{d}{dt}(\|\Phi\|_{2}^{2}+\|\Gamma\|_{2}^{2})+\|\nabla\Phi\|_{2}^{2}+\|\nabla\Gamma\|_{2}^{2}\nonumber\\
&\leq&C\|\frac{a}{r}\|_{\infty}(\|\nabla \Pi\|_{2}+\|\nabla \omega^{\theta}\|_{2}+\|\Gamma\|_{2}+\|\partial_{r}\omega^{3}-\partial_{3}\omega^{r}\|_{2})(\|\nabla\Gamma\|_{2}+\|\nabla\Phi\|_{2})\nonumber\\
&&+C_0  \|u^{\theta}\|_{3}\|\nabla\Gamma\|_{2}\|\nabla\Phi\|_{2}.
\end{eqnarray}
\end{lem}
\begin{proof}
Multiplying the equation (\ref{1.6}) by $(\Phi,\Gamma)$ respectively,   we have
\begin{eqnarray*}
0&=&\int_{\R^{3}}(\partial_{t}\Phi+\mathbf{u}\cdot\nabla\Phi+\frac{1}{r}\partial_{3}(\frac{1}{\rho}(\Delta-\frac{1}{r^{2}})u^{\theta})-(\omega^{r}\partial_{r}+\omega^{3}\partial_{3})\frac{u^{r}}{r})\cdot \Phi  ~dx\\
&&+\int_{\R^{3}}(\partial_{t}\Gamma+\mathbf{u}\cdot\nabla\Gamma-\frac{1}{r}\partial_{3}(\frac{1}{\rho}((\Delta-\frac{1}{r^{2}})u^{r}-\partial_{r}\Pi))+\frac{1}{r}\partial_{r}(\frac{1}{\rho}(\Delta u^{3}-\partial_{3}\Pi))+2\frac{u^{\theta}}{r}\Phi)\cdot\Gamma ~dx\\
&:=&\frac{1}{2}\frac{d}{dt}(\|\Phi\|_{2}^{2}+\|\Gamma\|_{2}^{2})+I_{1}+I_{2}+I_{3}.
\end{eqnarray*}
Then, using integration by parts (\ref{1.4}) and Lemma \ref{lem1}, we have
\begin{eqnarray*}
I_{1}&=&\int_{\R^{3}}\frac{1}{r}\partial_{3}((1+a)(\Delta-\frac{1}{r^{2}})u^{\theta})\cdot\Phi~dx\\
&=&\int_{\R^{3}}\frac{1}{r}\partial_{3}(\Delta-\frac{1}{r^{2}})u^{\theta})\cdot\Phi- \frac{a}{r}(\Delta-\frac{1}{r^{2}})u^{\theta})\cdot \partial_{3}\Phi~dx\\
&=&\int_{\R^{3}}-(\Delta+\frac{2}{r}\partial_{r})\Phi\cdot\Phi-\frac{a}{r}(\partial_{r}\omega^{3}-\partial_{3}\omega^{r})\cdot \partial_{3}\Phi~dx\\
&=&\|\nabla\Phi\|_{2}^{2}-2\pi\int_{\R}\int_{0}^{\infty} \partial_{r}(\Phi)^{2}~drdx_{3}-\int_{\R^{3}}\frac{a}{r}(\partial_{r}\omega^{3}-\partial_{3}\omega^{r})\cdot \partial_{3}\Phi~dx \\
&=&\|\nabla\Phi\|_{2}^{2}+2\pi\int_{\R}\Phi^{2}|_{r=0}~drdx_{3}-\int_{\R^{3}}\frac{a}{r}(\partial_{r}\omega^{3}-\partial_{3}\omega^{r})\cdot \partial_{3}\Phi~dx\\
&\geq&\|\nabla\Phi\|_{2}^{2}-\int_{\R^{3}}\frac{a}{r}(\partial_{r}\omega^{3}-\partial_{3}\omega^{r})\cdot \partial_{3}\Phi~dx.
\end{eqnarray*}
Similarly,  since the assumption $a|_{r=0}=0$,  we get
\begin{eqnarray*}
I_{2}&=&\int_{\R^{3}}\left(-\frac{1}{r}\partial_{3}((1+a)((\Delta-\frac{1}{r^{2}})u^{r}-\partial_{r}\Pi))+\frac{1}{r}\partial_{r}((1+a)(\Delta u^{3}-\partial_{3}\Pi))\right)\cdot\Gamma~dx \\
&=&\int_{\R^{3}}\left( -\frac{1}{r}\partial_{3}((1+a)(\partial_{3}\omega^{\theta}-\partial_{r}\Pi))-\frac{1}{r}\partial_{r}((1+a)(\partial_{r}\omega^{\theta}+\Gamma+\partial_{3}\Pi))\right)\cdot\Gamma~dx\\
&=&\int_{\R^{3}} -\frac{1}{r}(\Delta-\frac{1}{r^{2}})\omega^{\theta}\cdot\Gamma+\frac{a}{r}(\partial_{3}\omega^{\theta}-\partial_{r}\Pi)\cdot\partial_{3}\Gamma+\frac{a}{r}(\partial_{r}\omega^{\theta}+\Gamma+\partial_{3}\Pi)\partial_{r}\Gamma ~dx\\
&=&\int_{\R^{3}} -(\Delta+\frac{2}{r}\partial_{r})\Gamma\cdot\Gamma+\frac{a}{r}(\partial_{3}\omega^{\theta}-\partial_{r}\Pi)\cdot\partial_{3}\Gamma+\frac{a}{r}(\partial_{r}\omega^{\theta}+\Gamma+\partial_{3}\Pi)\partial_{r}\Gamma~dx\\
&=& \|\nabla\Gamma\|_{2}^{2}-2\pi\int_{\R}\int_{0}^{\infty} \partial_{r}(\Gamma)^{2}~drdx_{3}+\int_{\R^{3}} \frac{a}{r}(\partial_{3}\omega^{\theta}-\partial_{r}\Pi)\cdot\partial_{3}\Gamma+\frac{a}{r}(\partial_{r}\omega^{\theta}+\Gamma+\partial_{3}\Pi)\partial_{r}\Gamma~dx\\
&=& \|\nabla\Gamma\|_{2}^{2}+2\pi\int_{\R} \Gamma^{2}|_{r=0}~dx_{3}+\int_{\R^{3}} \frac{a}{r}(\partial_{3}\omega^{\theta}-\partial_{r}\Pi)\cdot\partial_{3}\Gamma+\frac{a}{r}(\partial_{r}\omega^{\theta}+\Gamma+\partial_{3}\Pi)\partial_{r}\Gamma~dx\\
&\geq& \|\nabla\Gamma\|_{2}^{2}+\int_{\R^{3}} \frac{a}{r}(\partial_{3}\omega^{\theta}-\partial_{r}\Pi)\cdot\partial_{3}\Gamma+\frac{a}{r}(\partial_{r}\omega^{\theta}+\Gamma+\partial_{3}\Pi)\partial_{r}\Gamma~dx.
\end{eqnarray*}
And using the similar calculus in \cite{H.Chen},  (\ref{2.2}) and Sobolev-Hardy inequality in Proposition \ref{prop}, we have
\begin{eqnarray*}
|I_{3}|&=&|\int_{\R^{3}}-(\omega^{r}\partial_{r}+\omega^{3}\partial_{3})\frac{u^{r}}{r}\cdot \Phi+2\frac{u^{\theta}}{r}\Phi\cdot\Gamma ~dx| \\
&=&|2\pi \int_{\R}\int^\infty_0(\partial_{3}u^{\theta}\partial_{r}\frac{u^{r}}{r}\Phi-\frac{\partial_{r}(ru^{\theta})}{r}\partial_{3} \frac{u^{r}}{r}\Phi)+2\frac{u^{\theta}}{r}\Phi\cdot\Gamma~rdrdx_{3}|\\
&=&|\int_{\R^{3}}-u^{\theta}(\partial_{3}\partial_{r}\frac{u^{r}}{r}\Phi+\partial_{r}\frac{u^{r}}{r}\partial_{3}\Phi)~dx
+\int_{\R^{3}}u^{\theta}(\partial_{r}\partial_{3} \frac{u^{r}}{r}\Phi+\partial_{3} \frac{u^{r}}{r}\partial_{r}\Phi)+2\frac{u^{\theta}}{r}\Phi\cdot\Gamma~dx|\\
&=&| \int_{\R^{3}}u^{\theta}(-\partial_{r}\frac{u^{r}}{r}\partial_{3}\Phi+\partial_{3} \frac{u^{r}}{r}\partial_{r}\Phi)+2\frac{u^{\theta}}{r}\Phi\cdot\Gamma~dx|\\
&\leq& \|u^{\theta}\|_{3}(\|\partial_{r}\frac{u^{r}}{r}\|_{6}\|\partial_{3}\Phi\|_{2}+\|\partial_{3}\frac{u^{r}}{r}\|_{6}\|\partial_{r}\Phi\|_{2})+2 \|u^{\theta}\|_{3}\|\frac{\Phi}{r^{\frac{1}{2}}}\|_{3}\|\frac{\Gamma}{r^{\frac{1}{2}}}\|_{3} \\
&\leq&C\|u^{\theta}\|_{3}\|\nabla\Gamma\|_{2}\|\nabla\Phi\|_{2}.
\end{eqnarray*}
Combining the above estimate, we have (\ref{2.14}).
\end{proof}

\begin{lem}
Under the condtions in Lemma \ref{l2.6}, we obtain that for all $t\in [0,T]$,
\begin{equation}\label{2.17}
\|\frac{a}{r}\|_{L_{t}^{\infty,\infty}}\leq \|\frac{a_{0}}{r}\|_{\infty}\exp(t^{\frac{3}{4}}\|\Gamma\|_{L_{t}^{\infty,2}}^{\frac{1}{2}}\|\nabla\Gamma\|_{L_{t}^{2,2}}^{\frac{1}{2}}) .
\end{equation}
\end{lem}
\begin{proof}
It follows from the transport equation of (\ref{1.3}) that
\begin{equation}\label{2.19}
\partial_{t}a+\mathbf{u}\cdot \nabla a=0,
\end{equation}
and
\begin{equation}
\partial_{t}\frac{a}{r}+\mathbf{u}\cdot \nabla \frac{a}{r}+\frac{u^{r}}{r}\ \frac{a}{r}=0,
\end{equation}
which yields (\ref{2.17}) by applying (\ref{infty}).
\end{proof}

\section{Proof of the well-posedness part of Theorem \ref{thm}}
We are going to complete the proof of the well-posedness part of Theorem  \ref{thm} in this section. It is well known that if the initial data $(\rho_{0},\mathbf{u_{0}})$ satisfies
$$
0<m\leq \rho_{0}\leq M,\ \ \mathbf{u_{0}}\in H^{1},
$$
then
the system (\ref{1.1}) has a local unique solution $(\rho,\mathbf{u})$ on $[0, T_{*})$  satisfying (\ref{1.7}) (see \cite{Paicu} for instance).

We mollify the initial data $(\rho_{0},\mathbf{u_{0}})$. Let $J^{\epsilon}=\epsilon^{-3}J(\frac{r}{\epsilon},\frac{x_{3}}{\epsilon})$ be  mollifiers, with
$$
0\leq J\leq 1, \ \
\mathrm{supp} J \subset \{ 0\leq r \leq 2, -1\leq x_{3}\leq 1\},
$$
$$
J\equiv 1, ~\text{if} ~x\in \{ 0\leq r \leq \frac{1}{2}, -\frac{1}{2}\leq x_{3}\leq \frac{1}{2}\},\ \
\int J ~dx =1,
$$
and
\begin{equation}
\rho_{0}^{\epsilon}=J^{\epsilon}\ast\rho_{0}-(J^{\epsilon}\ast(\rho_{0}-1))(0,x_{3}),~\mathbf{u}_{0}^{\epsilon}=J^{\epsilon}\ast\mathbf{u}_{0}.
\end{equation}
Obviously, $\rho_{0}^{\epsilon},\mathbf{u}_{0}^{\epsilon}$ are still axisymmetric. we claim  that (\ref{1.1}) has a unique global smooth axisymmetric solution $(\rho^{\epsilon},\mathbf{u}^{\epsilon})$ with the initial data $(\rho_{0}^{\epsilon},\mathbf{u}_{0}^{\epsilon})$, provided that (\ref{1.8}) is satisfied. Then the global existence part of Theorem \ref{thm} follows from uniform estimates (\ref{den}), (\ref{3.10}), and a standard compactness argument.

There are some properties of the initial data $(\rho_{0}^{\epsilon},\mathbf{u}_{0}^{\epsilon})$. For the convenience of the reader,  we give the proof of this
lemma in the Appendix.
\begin{lem}\label{l3.1}
If $\epsilon$ is sufficient small, and $\rho_{0}$ satisfies
$0< m\leq\rho_{0}\leq M,
$
then
$$
\rho_{0}^{\epsilon}=1, ~\text{if}~r=0,
$$
$$
0<\frac{m}{2}\leq\rho_{0}^{\epsilon}\leq \frac{M}{2},
$$
$$
|\rho_{0}^{\epsilon}-1|\leq C\|\frac{\rho_{0}-1}{r}\|_{\infty}\  r .
$$
\end{lem}

It is easy to show that $a_{0},\mathbf{u_{0}}\in H^{\infty}$. From the local well-posedness result in \cite{Dachin3} (Corollary 0.8) and \cite{Chae},  it ensures that the system admits a unique axisymetric solution $(a^{\epsilon},\mathbf{u}^{\epsilon},\nabla \Pi^{\epsilon})$ of the equations derived from (\ref{1.1})
\begin{equation*}
\left\{
\begin{array}{l}
\partial_{t}a+ \mathrm{div}(a \mathbf{u})=0,\\
\partial_{t}\mathbf{u}+ \mathbf{u}\cdot\nabla\mathbf{u}-(1+a)\Delta \mathbf{u}+(1+a)\nabla \Pi=\mathbf{0},\\
 \mathrm{div} \mathbf{u}=0,\\
(a,\mathbf{u})|_{t=0}=(a_{0},\mathbf{u_{0}})
\end{array}
\right.~(t,x)\in\R^+\times\R^{3},
\end{equation*}
in $[0,T_{*}^{\epsilon})$. And for any $T^{\epsilon}<T_{*}^{\epsilon}$, the solution satisfyies
\begin{eqnarray*}
a^{\epsilon}\in \mathcal{C}([0,T^{\epsilon}];H^{s});\ \mathbf{u}^{\epsilon} \in \mathcal{C}([0,T^{\epsilon}];H^{s})\cap \tilde{L}^{1}(0,T^{\epsilon};H^{s+2}),\\
\nabla\Pi^{\epsilon}\in \tilde {L}^{1}(0,T^{\epsilon};H^{s}),~s>\frac{5}{2}.
\end{eqnarray*}

Then, we will show that the maximal existence time $T_{*}^{\epsilon}=\infty$ as follows,  provided (\ref{1.8}) is satisfied.

Without loss of generality, we denote $\rho=\rho^{\epsilon},\mathbf{u}=\mathbf{u}^{\epsilon},\Pi=\Pi^{\epsilon}$, and so on. And we assume $T_{*}<\infty$.

\begin{lem} We claim that $a|_{r=0}=0.$\end{lem}
\begin{proof} We can define the unique trajectory $\boldsymbol{\chi}(t,x)$ of $\mathbf{u}(t,x)$ by
$$
\partial_{t}\boldsymbol{\chi}(t,x)=\mathbf{u}(t,\boldsymbol{\chi}(t,x)),~\boldsymbol{\chi}(0,x)=x.
$$
Since $u^r|_{r=0}=u^\theta|_{r=0}=0$, we have that $\boldsymbol{\chi}(t,x)=(0,0,\chi^{3}(t,x_{3})$  satisfying
$$
\partial_{t}\chi^{3}(t,x_{3})=u^{3}(t,\chi^{3}(t,x_{3})),~\chi^{3}(0,x_{3})=x_{3}.
$$
is the trajectory from the initial point $(0,0,x_{3})$. Therefore, by (\ref{2.19}), there exists a trajectory $\boldsymbol{\chi}(t,x)=(0,0,\chi^{3}(t,x^{3}))$,
$$
a(t,x)|_{r=0}=a(t,\boldsymbol{\chi}(t,x))=a(0,0,x_{3})=a_{0}(x)|_{r=0}=0.
$$
\end{proof}

\begin{lem} \label{lem3} There exists $C_{1}$, such that if $T_{*}>N \triangleq C_{1}\|\mathbf{u_{0}}\|_{2}^{4}$, then
\begin{equation}
\|\nabla \mathbf{u}(t)\|_{2}^{2}+\int_{N}^{t}\|\nabla^{2}\mathbf{u}(\tau)\|_{2}^{2}+\|\nabla\Pi(\tau)\|_{2}^{2}~d\tau\leq C\frac{1}{\|\mathbf{u}_{0}\|_{2}^{2}},~ \forall t\in[N,T_{*}).
\end{equation}
\end{lem}
\begin{proof}
By (\ref{2.3}), there exists a positive constant $K_{1}$, such that
$$\underset{t\in[0,T^*)}{\sup}\|\mathbf{u}(t)\|_{2}^{2}+\int_{0}^{T^*} \|\nabla \mathbf{u}\|_{2}^{2}\leq K_{1} \|\mathbf{u_{0}}\|_{2}^{2}.$$
There exists a time $t_{0}\in (0,N)$, such that
$$\|\nabla\mathbf{u}(t_{0})\|_{2}^{2}\leq K_{1}\frac{\|\mathbf{u_{0}}\|_{2}^{2}}{N} .
$$

Thus
$$
\|\mathbf{u}(t_{0})\|_{2}^{2}\|\nabla\mathbf{u}(t_{0})\|_{2}^{2}\leq \frac{K_{1}^2}{C_{1}}.
$$

From the similar argument as that in the proof of \textit{a priori} estimate revealed in Lemma 2.2 in \cite{Hammadi}, one can easily obtain that
$$
\frac{d}{dt}\|\nabla\mathbf{u}\|_{2}^{2}+\|\sqrt{\rho}\mathbf{u}_{t}\|_{2}^{2}+\|\nabla^{2}\mathbf{u}\|_{2}^{2}+\|\nabla\Pi\|_{2}^{2}\leq K_{2}\|\mathbf{u}\|_{2}\|\nabla\mathbf{u}\|_{2}\|\nabla^{2}\mathbf{u}\|_{2}^{2}, ~~t\in[t_{0},T_{*}),
$$
We can pick that $C_{1}>4K_{1}^2K_{2}^2$ is sufficiently large. By using the continuous method, it is evidence to show that
$$\|\nabla \mathbf{u}(t)\|_{2}^{2}+\int_{t_{0}}^{t}\|\nabla^{2}\mathbf{u}(\tau)\|_{2}^{2}+\|\nabla\Pi(\tau)\|_{2}^{2}~d\tau\leq \|\nabla \mathbf{u}(t_{0})\|_{2}^{2} \leq \frac{K_1}{C_1\|\mathbf{u}_{0}\|_{2}^{2}}, ~t\in[t_{0},T_{*}).
$$
\end{proof}

Now, we can deduce the contradiction by the continuous method.

Pick $C_{2}=4C_0$ where $C_0$ is a positive constant in (\ref{2.14}). We assume that there exists a maximal time $T_{0}\leq\min \{T_{*},N\}$, such that for $t\in [0,T_{0})$,
\begin{equation}\label{3.5}
\left\{\begin{array}{lll}
\|\Gamma\|_{L_{t}^{\infty,2}}^{2}+\|\Phi\|_{L_{t}^{\infty,2}}^{2}+\|\nabla\Gamma\|_{L_{t}^{2,2}}^{2} \leq 2 (\|\Gamma_{0}\|_{2}^{2}+\|\Phi_{0}\|_{2}^{2}),\\
\|u^{\theta}\|_{L_{t}^{\infty,3}}\leq \frac{1}{C_{2}}.
\end{array}\right.
\end{equation}

Then from (\ref{2.14}) and H\"{o}lder's inequality, we obtain that for all $t\in [0,T_{0})$,
\begin{eqnarray*}
&&\frac{d}{dt}(\|\Phi\|_{2}^{2}+\|\Gamma\|_{2}^{2})+\|\tilde\nabla\Phi\|_{2}^{2}+\|\tilde\nabla\Gamma\|_{2}^{2}\\
&\leq & \frac{1}{2}(\|\tilde\nabla\Phi\|_{2}^{2}+\|\tilde\nabla\Gamma\|_{2}^{2})+C\|\frac{a}{r}\|_{\infty}^{2}
(\|\nabla \Pi\|_{2}+\|\nabla \omega^{\theta}\|_{2}+\|\Gamma\|_{2}+\|\partial_{r}\omega^{3}-\partial_{3}\omega^{r}\|_{2})^{2}.
\end{eqnarray*}
Thus from (\ref{1.8}), (\ref{priori}), (\ref{2.17}), and $t<T_{0}\leq N=C\|\mathbf{u}_{0}\|_{2}^{4}$, we have
\begin{eqnarray*}
&&\|\Phi\|_{L_{t}^{\infty,2}}^{2}+\|\Gamma\|_{L_{t}^{\infty,2}}^{2}+\|\tilde\nabla\Phi\|_{L_{t}^{2,2}}^{2}+\|\tilde\nabla\Gamma\|_{L_{t}^{2,2}}^{2}\\
&\leq&\|\Gamma_{0}\|_{2}^{2}+\|\Phi_{0}\|_{2}^{2}+C \|\frac{a_{0}}{r}\|_{\infty}^{2}\exp(C N^{\frac{3}{4}} (\|\Gamma_{0}\|_{2}+\|\Phi_{0}\|_{2}))\\
&&\times (\|(u_{0}^{\theta})^{2}\|_{2}^{2}+\|\nabla \mathbf{u}_{0}\|_{2}^{2}+\|\mathbf{u}_{0}\|_{2}^{4}(\|\Gamma_{0}\|_{2}^{2}+\|\Phi_{0}\|_{2}^{2})) \exp( C\|\mathbf{u}_{0}\|_{2}^{3}(\|\Gamma_{0}\|_{2}+\|\Phi_{0}\|_{2}))\\
&\leq& \|\Gamma_{0}\|_{2}^{2}+\|\Phi_{0}\|_{2}^{2}+C\|\frac{a_{0}}{r}\|_{\infty}^{2}(\|(u_{0}^{\theta})^{2}\|_{2}^{2}+\|\nabla \mathbf{u}_{0}\|_{2}^{2}+\|\mathbf{u}_{0}\|_{2}^{4}(\|\Gamma_{0}\|_{2}^{2}+\|\Phi_{0}\|_{2}^{2}))\\
&&\times \exp( C\|\mathbf{u}_{0}\|_{2}^{3}(\|\Gamma_{0}\|_{2}+\|\Phi_{0}\|_{2}))\\
&\leq& \frac{3}{2}(\|\Gamma_{0}\|_{2}^{2}+\|\Phi_{0}\|_{2}^{2}).
\end{eqnarray*}
Multiplying the equation $(\ref{1.3})_{3}$ by $(u^{\theta})^{2}$, and using integration by parts, we have
\begin{eqnarray*}
&&\frac{1}{3}\left\|\sqrt{\rho}(u^{\theta})^{\frac{3}{2}}\right\|_{2}^{2}+\frac{8}{9}\|\nabla(|u^{\theta}|^{\frac{3}{2}})\|_{2}^{2}+\|r^{-2} (u^{\theta})^{3}\|_{1}\\
&=&-\int_{\R^{3}}\rho \frac{u^{r}}{r}~|u^{\theta}|^{3}dx\\
&\leq& \|\frac{u^{r}}{r}\|_{\infty} \left\|\sqrt{\rho}(u^{\theta})^{\frac{3}{2}}\right\|_{2}^{2}.
\end{eqnarray*}
Thus by (\ref{infty}) and (\ref{1.9}), we obtain
\begin{eqnarray*}
\|u^{\theta}\|_{L_{t}^{\infty,3}}&\leq&C \|u_{0}^{\theta}\|_{3}\exp(C\|\frac{u^{r}}{r}\|_{L_{t}^{1,\infty}})\\
&\leq&C \|u_{0}^{\theta}\|_{3}\exp(C N^{\frac{3}{4}}(\|\Gamma_{0}\|_{2}+\|\Phi_{0}\|_{2}))\\
&\leq& \frac{1}{2C_{2}}.
\end{eqnarray*}

By applying the continuous method,   we have the conclusion that $T_{0}=\min \{T_{*},N\}$, and (\ref{3.5}) holds for any $t\in [0,T_{0})$.

Moreover, by combining (\ref{priori}) and (\ref{3.5}), we have for any $t\in [0,T_{0})$,
\begin{equation}\label{3.9}
\|\nabla \mathbf{u}\|_{L_{t}^{\infty,2}}^{2}+\|\nabla^{2} \mathbf{u}\|_{L_{t}^{2,2}}^{2}+\|\nabla\Pi\|_{L_{t}^{2,2}}^{2}\leq C\mathcal{G}.
\end{equation}
where
$$
\mathcal{G}=C(\|(u_{0}^{\theta})^{2}\|_{2}^{2}+\|\nabla \mathbf{u}_{0}\|_{2}^{2}+\|\mathbf{u}_{0}\|_{2}^{4}(\|\Gamma_{0}\|_{2}^{2}+\|\Phi_{0}\|_{2}^{2})) \exp( C\|\mathbf{u}_{0}\|_{2}^{3}(\|\Gamma_{0}\|_{2}^{2}+\|\Phi_{0}\|_{2}^{2})).
$$

Recall Lemma \ref{lem3} and the conversation law (\ref{2.3}). We have for any $t<T_{*}$,
\begin{equation}\label{3.10}
\|\mathbf{u}\|_{L^{\infty}((0,t);H^{1})}^{2}+\int_{0}^{t}\|\nabla\mathbf{u}(\tau)\|_{H^{1}}^{2}+\|\nabla\Pi(\tau)\|_{2}^{2}~d\tau\leq C\mathcal{G}+C\frac{1}{\|\mathbf{u}_{0}\|_{2}^{2}}.
\end{equation}

Thanks to the calculus in \cite{Kim} and the blow up criteria (See Proposition 0.6 in \cite{Dachin3}, for instance), we deduce the contraction with the fact that $T_{*}$ is the blow up time of the solution. Thus, we obtain that $T_{*}=\infty$,
and finish the proof of well-posedness part of Theorem \ref{thm}.
$\hfill\Box$

\section{Proof of the decay estimates part of Theorem \ref{thm}}

When $\mathbf{u}_{0}\in L^{1}(\R^{3})$, from the proof in \cite{Hammadi} (Section 3), we can obtain the  decay estimates (\ref{decayA}) and omit the details.

\noindent\textbf{Proof of the decay estimate (\ref{1.10}).}

$\bullet$ The decay estimate of $\|ru^{\theta}\|_{2}^{2}.$

From (\ref{1.3}), we have
\begin{equation}
\rho \partial_{t}(ru^{\theta})+\rho \mathbf{u}\cdot\nabla (ru^{\theta})-(\Delta-\frac{2}{r}\partial_{r})(ru^{\theta})=0.\label{ruthe}
\end{equation}
Multiply the equation  by $|ru^{\theta}|^{p-2}ru^{\theta},1<p<\infty$, and using integration by parts, we have
\begin{equation*}
\|ru^{\theta}(t)\|_{p}\leq \|ru_{0}^{\theta}\|_{p}.
\end{equation*}
Then, one can easily obtain that
\begin{equation}
\|ru^{\theta}(t)\|_{1}\leq \|ru_{0}^{\theta}\|_{1}.\label{utheL1}
\end{equation}

Moreover, if $\|ru_{0}^{\theta}\|_{L^{1}\cap L^{2}}\leq C$, from (\ref{ruthe}), we have
\begin{equation}
\frac{1}{2}\frac{d}{dt}\|\sqrt{\rho}ru^{\theta}\|_{2}^{2}+\|\nabla (ru^{\theta})\|_{2}^{2}=0.\label{3.14}
\end{equation}
By the Sobolev embedding theorem and (\ref{utheL1}), one obtain
\begin{eqnarray}
\|ru^{\theta}\|_{2}&\leq& C \|ru^{\theta}\|_{1}^{\frac{2}{5}}\|\nabla (ru^{\theta})\|_{2}^{\frac{3}{5}}\nonumber\\
&\leq& C\|ru_{0}^{\theta}\|_{1}^{\frac{2}{5}}\|\nabla (ru^{\theta})\|_{2}^{\frac{3}{5}}\nonumber\\
&\leq& C\|\nabla (ru^{\theta})\|_{2}^{\frac{3}{5}}.\label{3.15-2}
\end{eqnarray}
From (\ref{3.14})-(\ref{3.15-2}), we have
\begin{equation*}
\frac{d}{dt}\|\sqrt{\rho}ru^{\theta}\|_{2}^{2}\leq- C(\|ru^{\theta}\|_{2}^{2})^{\frac{5}{3}}\leq- C(\|\sqrt{\rho}ru^{\theta}\|_{2}^{2})^{\frac{5}{3}},
\end{equation*}
and
\begin{equation}
\|ru^{\theta}\|_{2}^{2}\leq  C\leq\|\sqrt{\rho}ru^{\theta}\|_{2}^{2}\leq  C\langle t\rangle^{-\frac{3}{2}}.\label{3.16}
\end{equation}

$\bullet$ The decay estimate of $\|u^{\theta}(t)\|_{2}^{2}.$

Multiply the equation $(\ref{1.3})_{3}$ by $u^{\theta}$, we have
\begin{eqnarray}\label{3.15}
\frac{1}{2}\frac{d}{dt}\|\sqrt{\rho}u^{\theta}\|_{2}^{2}+\|\nabla u^{\theta}\|_{2}^{2}+\|\frac{u^{\theta}}{r}\|_{2}^{2}&=&-\int_{\R^{3}}\rho \frac{u^{r}}{r}(u^{\theta})^{2} ~dx\nonumber\\
&\leq&C\|\frac{u^{r}}{r}\|_{2}^{4}\|u^{\theta}\|_{2}^{2}+\frac{1}{2}\|\nabla u^{\theta}\|_{2}^{2}.
\end{eqnarray}
Applying the decay estimates (\ref{decayA}), we have
\begin{equation*}
\frac{d}{dt}\|\sqrt{\rho}u^{\theta}\|_{2}^{2}+\|\frac{u^{\theta}}{r}\|_{2}^{2}\leq C\|\frac{u^{r}}{r}\|_{2}^{4}\|u^{\theta}\|_{2}^{2} \leq C\langle t\rangle^{-\frac{13}{2}}.
\end{equation*}
Set $S(t)=\{x|r\leq M^{-\frac{1}{2}}g(t)^{-1}\},g(t)=\sqrt{\gamma} \ \ (1+t)^{-\frac{1}{2}},\gamma>\frac{5}{2}.$ From (\ref{3.16}), we get
\begin{eqnarray*}
\frac{d}{dt}\|\sqrt{\rho}u^{\theta}(t)\|_{2}^{2}+g(t)^{2}\|\sqrt{\rho}u^{\theta}(t)\|_{2}^{2}
&=&\frac{d}{dt}\|\sqrt{\rho}u^{\theta}\|_{2}^{2}+g(t)^{2}(\int_{S(t)}\rho|u^{\theta}|^{2} ~dx+\int_{S(t)^{c}}\rho|u^{\theta}|^{2} ~dx)\\
&\leq&\frac{d}{dt}\|\sqrt{\rho}u^{\theta}\|_{2}^{2}+\int_{S(t)}|\frac{u^{\theta}}{r}|^{2}~dx+g(t)^{2}\int_{S(t)^{c}}\rho|u^{\theta}|^{2} ~dx\\
 &\leq&C\langle t\rangle^{-\frac{13}{2}}+Mg(t)^{2}\int_{S(t)^{c}}\frac{1}{r^{2}}|ru^{\theta}|^{2}~dx\\
&\leq&C\langle t\rangle^{-\frac{13}{2}}+M^2g(t)^{4}\|ru^{\theta}\|_{2}^{2}\\
&\leq&C\langle t\rangle^{-\frac{7}{2}},
\end{eqnarray*}
and
\begin{equation*}
e^{\int_{0}^{t}g(\tau)^{2}d\tau}\|\sqrt{\rho}u^{\theta}(t)\|_{2}^{2}\leq\|\sqrt{\rho_{0}}u_{0}^{\theta}\|_{2}^{2}+C\int_{0}^{t}e^{\int_{0}^{\tau}g(s)^{2}ds} \langle \tau\rangle^{-\frac{7}{2}}  ~d\tau.
\end{equation*}
Since $e^{\int_{0}^{t}g(\tau)^{2}d\tau}\approx \langle t\rangle ^{\gamma},\gamma>\frac{5}{2}$, we have
\begin{equation*}
\langle t\rangle ^{\gamma}\|\sqrt{\rho}u^{\theta}(t)\|_{2}^{2}\leq C\|\sqrt{\rho_{0}}u_{0}^{\theta}\|_{2}^{2}+C \langle t\rangle ^{\gamma-\frac{5}{2}},
\end{equation*}
and
\begin{equation}
\|u^{\theta}(t)\|_{2}^{2}\leq C\|\sqrt{\rho}u^{\theta}(t)\|_{2}^{2}\leq C\langle t\rangle ^{-\frac{5}{2}}.
\end{equation}
$\bullet$ The decay estimate of $\|\nabla (u^{\theta}\mathbf{e}_{\theta})\|_{2}^{2}.$

We notice that
\begin{equation}\label{4.01}
\|\nabla (u^{\theta}\mathbf{e}_{\theta})\|_{2}^{2}=\|\nabla u^{\theta}\|_{2}^{2}+\|\frac{u^{\theta}}{r}\|_{2}^{2}=\|\omega^{r}\|_{2}^{2}+\|\omega^{3}\|_{2}^{2},
\end{equation}
\begin{equation}
\Delta(u^{\theta}\mathbf{e}_{\theta})=(\Delta-\frac{1}{r^{2}})u^{\theta}\mathbf{e}_{\theta}.
\end{equation}
And applying (\ref{1.4}) and (\ref{27}), we obtain, directly from $(\ref{1.3})_{3}$, that
\begin{eqnarray}\label{4.18}
\|(\Delta-\frac{1}{r^{2}})u^{\theta}\|_{2}&\leq& \|\rho u_{t}^{\theta}\|_{2}+\|\rho(\mathbf{u}\cdot\nabla +\frac{u^{r}}{r})u^{\theta}\|_{2}\nonumber\\
&\leq&C\|\sqrt{\rho}u_{t}^{\theta}\|_{2}+C\|u^{r}\omega^{3}\|_{2}+C\|u^{3}\omega^{r}\|_{2}\nonumber \\
&\leq&C\|\sqrt{\rho}u_{t}^{\theta}\|_{2}+C\|\nabla\mathbf{u}\|_{2}^{\frac{3}{2}}(\|\nabla\omega^{r}\|_{2}+\|\nabla \omega^{3}\|_{2})^{\frac{1}{2}}\nonumber\\
&\leq&C\|\sqrt{\rho}u_{t}^{\theta}\|_{2}+C\|\nabla\mathbf{u}\|_{2}^{\frac{3}{2}}\|(\Delta-\frac{1}{r^{2}})u^{\theta}\|_{2}^{\frac{1}{2}} \nonumber\\
&\leq&C\|\sqrt{\rho}u_{t}^{\theta}\|_{2}+C\|\nabla\mathbf{u}\|_{2}^{3}+\frac{1}{2} \|(\Delta-\frac{1}{r^{2}})u^{\theta}\|_{2}.
\end{eqnarray}

Set $s=\frac{t}{2}$. From (\ref{decayA}) and (\ref{3.15}), apply the Grownwall inequality, we have
\begin{eqnarray}\label{3.17}
\|\sqrt{\rho}u^{\theta}(t)\|_{2}^{2}+\int_{s}^{t}\|\nabla u^{\theta}(\tau)\|_{2}^{2}+\|\frac{u^{\theta}(\tau)}{r}\|_{2}^{2} ~d\tau &\leq& C\|\sqrt{\rho}u^{\theta}(s)\|_{2}^{2}\exp(C\int_{s}^{t}\|\nabla \mathbf{u}(\tau)\|_{2}^{4}~d\tau)\nonumber\\
&\leq& C\|u^{\theta}(s)\|_{2}^{2}\nonumber\\
&\leq& C\langle t\rangle ^{-\frac{5}{2}}.
\end{eqnarray}

Multiplying the equation $(\ref{1.3})_{3}$ by $u_{t}^{\theta}$, using integration by parts, we have
\begin{eqnarray}\label{3.181}
\frac{d}{dt}(\|\nabla u^{\theta}\|_{2}^{2}+\|\frac{u^{\theta}}{r}\|_{2}^{2})+\|\sqrt{\rho}u_{t}^{\theta}\|_{2}^{2}&=&-\int_{\R^{3}}\rho( \mathbf{u}\cdot\nabla u^{\theta}+\frac{u^{r}}{r}u^{\theta})u_{t}^{\theta}~dx\nonumber\\
&=&-\int_{\R^{3}}\rho(u^{r}\omega^{3}-u^{3}\omega^{r} )u_{t}^{\theta}~dx\nonumber\\
&\leq& C_{\delta} \|\nabla\mathbf{u}\|_{2}^{4}(\|\omega^{r}\|_{2}^{2}+\|\omega^{3}\|_{2}^{2})\nonumber\\
&&+\delta(\|\nabla\omega^{r}\|_{2}^{2}+\|\nabla\omega^{3}\|_{2}^{2}+\|\sqrt{\rho}u_{t}^{\theta}\|_{2}^{2}).
\end{eqnarray}

From (\ref{2.15}), we have
\begin{eqnarray}\label{3.19}
&&\frac{1}{2}\frac{d}{dt}(\|\omega^{r}\|_{2}^{2}+\|\omega^{3}\|_{2}^{2})+\int_{\R^{3}}\frac{1}{\rho}(\partial_{3}\omega^{r}-\partial_{r}\omega^{3})^{2}~dx\nonumber\\
&=&\int_{\R^{3}}~(\omega^{r}\partial_{r}+\omega^{3}\partial_{3})u^{r}\omega^{r}+(\omega^{r}\partial_{r}+\omega^{3}\partial_{3})u^{3}\omega^{3}~dx\nonumber\\
&\leq& C_{\delta}\|\nabla\mathbf{u}\|_{2}^{4}(\|\omega^{r}\|_{2}^{2}+\|\omega^{3}\|_{2}^{2})+\delta(\|\nabla\omega^{r}\|_{2}^{2}+\|\nabla\omega^{3}\|_{2}^{2}).
\end{eqnarray}

Let $f_{1}(t)=\|\omega^{r}(t)\|_{2}^{2}+\|\omega^{3}(t)\|_{2}^{2}$. From (\ref{3.17}), it satisfies that
\begin{equation}\label{3.20}
\int_{s}^{t}f_{1}(\tau)d\tau\leq C\langle t\rangle ^{-\frac{5}{2}}.
\end{equation}
Combine (\ref{3.181}) and (\ref{3.19}), applying (\ref{27}) (\ref{den}) and (\ref{4.01}), picking $\delta$ sufficiently small, we have
\begin{equation}\label{3.21}
\frac{d}{dt}f_{1}(t)+\|(\Delta-\frac{1}{r^{2}})u^{\theta}\|_{2}^{2}+\|\sqrt{\rho}u_{t}^{\theta}\|_{2}^{2} \leq C \|\nabla \mathbf{u}\|_{2}^{4}f_{1}(t).
\end{equation}
Multiplying the above inequality by $(t-s)$ leads to
\begin{equation*}
\frac{d}{dt}((t-s)f_{1}(t))\leq f_{1}(t)+C\|\nabla \mathbf{u}\|_{2}^{4} (t-s)f_{1}(t).
\end{equation*}
Applying Grownwall inequality, we have
\begin{equation*}
(t-s)f_{1}(t)\leq \int_{s}^{t}f_{1}(\tau)d\tau\exp(C\int_{s}^{t}\|\nabla\mathbf{u}(\tau)\|_{2}^{4}d\tau)\leq C\int_{s}^{t}f_{1}(\tau)d\tau.
\end{equation*}
Taking $s=\frac{t}{2}$, from (\ref{3.20}), we have
\begin{equation*}
f_{1}(t)\leq C t^{-1} \langle t\rangle^{-\frac{5}{2}},
\end{equation*}
thus
\begin{equation}
\|\nabla (u^{\theta}\mathbf{e}_{\theta})(t)\|_{2}^{2}= f_{1}(t)\leq C  \langle t\rangle^{-\frac{7}{2}}.
\end{equation}

$\bullet$ The decay estimates of $\|u_{t}^{\theta}(t)\|_{2}^{2}+\|(\Delta-\frac{1}{r^{2}}) u^{\theta}(t)\|_{2}^{2}$.

Applying Grownwall lemma to (\ref{3.21}) over $[s,t],s=\frac{t}{2}$, we have
\begin{equation}\label{4.16}
f_{1}(t)+\int_{s}^{t}\|(\Delta-\frac{1}{r^{2}})u^{\theta}\|_{2}^{2}+\|\sqrt{\rho}u_{t}^{\theta}\|_{2}^{2} d\tau \leq f_{1}(s)\exp (C\int_{s}^{t} \|\nabla\mathbf{u}\|_{2}^{4}d\tau)\leq C f_{1}(s)\leq C \langle t\rangle^{-\frac{7}{2}}.
\end{equation}
Applying (\ref{decayA}), we have
\begin{equation}\label{4.17}
\int_{s}^{t}(\tau-s)\|\mathbf{u}_{t}\|_{2}^{2}\leq C t \int_{s}^{t}\tau^{-1}\langle\tau\rangle^{-\frac{5}{2}}d\tau\leq C\langle t\rangle^{-\frac{3}{2}}.
\end{equation}

By taking $\partial_{t}$ to $(\ref{1.3})_{3}$, we have
\begin{equation}
\rho u_{tt}^{\theta}+\rho \mathbf{u}\cdot\nabla u_{t}^{\theta}-(\Delta-\frac{1}{r^{2}})u_{t}^{\theta}=-\rho_{t} u_{t}^{\theta}-\partial_{t}(\rho\mathbf{u})\cdot\nabla u^{\theta}-\partial_{t} (\rho\frac{u^{r}}{r}u^{\theta}).
\end{equation}
Taking $L^{2}$ inner product of the above equation with $u_{t}^{\theta}$ and using the transport equation $(\ref{1.3})_{1}$, we have
\begin{eqnarray}
\frac{1}{2}\frac{d}{dt}\|\sqrt{\rho}u_{t}^{\theta}\|_{2}^{2}+\|\nabla u_{t}^{\theta}\|_{2}^{2}+\|\frac{u_{t}^{\theta}}{r}\|_{2}^{2} &=& \int_{\R^{3}} \mathrm{div} (\rho \mathbf{u})(u_{t}^{\theta}+\mathbf{u}\cdot\nabla u^{\theta}+u^{r}\frac{u^{\theta}}{r})u_{t}^{\theta}\nonumber\\
&&-\rho \mathbf{u}_{t}\cdot\nabla u^{\theta} u_{t}^{\theta}-\rho \frac{u_{t}^{r}u^{\theta}+u^{r}u_{t}^{\theta}}{r}u_{t}^{\theta}~dx \nonumber\\
&=& \int_{\R^{3}}\mathrm{div}(\rho \mathbf{u})(u_{t}^{\theta}+u^{3}\partial_{3}u^{\theta}+u^{r}\omega^{3})u_{t}^{\theta}-\rho u_{t}^{r}\omega^{3}u_{t}^{\theta} \nonumber\\
&&-\rho u_{t}^{3}\partial_{3}u^{\theta}u_{t}^{\theta}-\rho \frac{u^{r}}{r} u_{t}^{\theta}u_{t}^{\theta}~dx.
\end{eqnarray}
Using integration by parts,   H\"{o}lder inequality, Sobolev inequality, Cauchy inequality, (\ref{den}) and (\ref{4.18}), we have
\begin{eqnarray}
&&\frac{d}{dt}\|\sqrt{\rho}u_{t}^{\theta}\|_{2}^{2}+\|\nabla u_{t}^{\theta}\|_{2}^{2}+\|\frac{u_{t}^{\theta}}{r}\|_{2}^{2}
\nonumber\\
&\leq &C\|\nabla\mathbf{u}\|_{2}\|u_{t}^{\theta}\|_{3}(\|\nabla u_{t}^{\theta}\|_{2}+\|\frac{u_{t}^{\theta}}{r}\|_{2})+C\|\nabla\mathbf{u}\|_{2}^{2}(\|\nabla\omega^{r}\|_{2}+\|\nabla \omega^{3}\|_{2})\nonumber\\
&&\times\|\nabla u_{t}^{\theta}\|_{2}
+C\|\mathbf{u}_{t}\|_{2}(\|\omega^{r}\|_{3}+\|\omega^{3}\|_{3})\|\nabla u_{t}^{\theta}\|_{2}\nonumber \\
&\leq& C\|\nabla \mathbf{u}\|_{2}^{4}\|\sqrt{\rho}u_{t}^{\theta}\|_{2}^{2}+C\|\nabla \mathbf{u}\|_{2}^{10}+C(\|\nabla \mathbf{u}\|_{2}^{4}+\|\nabla \mathbf{u}\|_{2} \|u_{t}^{\theta}\|_{2})\|\mathbf{u}_{t}\|_{2}^{2}\nonumber\\
&&+\frac{1}{2}(\|\nabla u_{t}^{\theta}\|_{2}^{2}+\|\frac{u_{t}^{\theta}}{r}\|_{2}^{2}).
\end{eqnarray}
Multiplying the above inequality by $(t-s)$, and applying Grownwall inequality on $[s,t]$, we have
\begin{eqnarray}
&&(t-s)\|\sqrt{\rho}u_{t}^{\theta}(t)\|_{2}^{2}\nonumber\\
 &\leq& C(\int_{s}^{t}\|\sqrt{\rho}u_{t}^{\theta}(\tau)\|_{2}^{2}+(\tau-s)\|\nabla \mathbf{u}\|_{2}^{10}+(\tau-s)(\|\nabla \mathbf{u}\|_{2}^{4}+\|\nabla \mathbf{u}\|_{2} \|u_{t}^{\theta}\|_{2})\|\mathbf{u}_{t}\|_{2}^{2}d\tau )\nonumber\\
&&\times\exp(C\int_{s}^{t}\|\nabla\mathbf{u}\|_{2}^{4}d\tau).
\end{eqnarray}
Taking $s=\frac{t}{2},t>1$, applying (\ref{decayA}), (\ref{4.16}), and (\ref{4.17}), we have
\begin{eqnarray}
t\|\sqrt{\rho}u_{t}^{\theta}(t)\|_{2}^{2} &\leq& C(\int_{s}^{t}\|\sqrt{\rho}u_{t}^{\theta}(\tau)\|_{2}^{2}+(\tau-s)\|\nabla \mathbf{u}\|_{2}^{10}+(\tau-s)(\|\nabla \mathbf{u}\|_{2}^{4}+\|\nabla \mathbf{u}\|_{2} \|u_{t}^{\theta}\|_{2})\|\mathbf{u}_{t}\|_{2}^{2}d\tau )\nonumber\\
&\leq&C \left(\langle t\rangle ^{-\frac{7}{2}}+t \underset{\tau \in [s,t]}{\sup} \|\nabla\mathbf{u}(\tau)\|_{2}^{8}+\underset{\tau \in [s,t]}{\sup}(\|\nabla \mathbf{u}\|_{2}^{4}+\|\nabla \mathbf{u}\|_{2} \|u_{t}^{\theta}\|_{2}) \int_{s}^{t}(\tau-s)\|\mathbf{u}_{t}\|_{2}^{2} d\tau\right)\nonumber\\
&\leq&C \left(\langle t\rangle ^{-\frac{7}{2}}+ \langle t\rangle ^{-9} +(\langle t\rangle ^{-5}+\langle t\rangle ^{-\frac{5}{4}}t^{-\frac{1}{2}}\langle t\rangle ^{-\frac{5}{4}})\langle t\rangle ^{-\frac{3}{2}}  \right)\nonumber\\
&\leq&C \langle t\rangle ^{-\frac{7}{2}}.
\end{eqnarray}
Thus by (\ref{4.18}), we have
\begin{equation}
\|u_{t}^{\theta}(t)\|_{2}^{2}+\|(\Delta-\frac{1}{r^{2}})u^{\theta}(t)\|_{2}^{2} \leq C t^{-1}\langle t\rangle ^{-\frac{7}{2}}, ~\forall t>0.
\end{equation}
Therefore the result (\ref{1.10}) can be directly derived, and Theorem \ref{thm} is proved.
 $\hfill\Box$

\section{Appendix}
\noindent\textbf{Proof of Lemma \ref{l2.4}.}

Multiplying the equation $(\ref{1.3})_{3}$ by $(u^{\theta})^{3}$,  using integration by parts, applying H\"{o}lder inequality, Sobolev inequality, Cauchy inequality and (\ref{2.2}), we have
\begin{eqnarray*}
&&\frac{1}{4}\frac{d}{dt}\left\|\sqrt{\rho}(u^{\theta})^{2}\right\|_{2}^{2}+\frac{3}{4}\left\|\nabla (u^{\theta})^{2}\right\|_{2}^{2}
+\left\|\frac{(u^{\theta})^{2}}{r}\right\|_{2}^{2}\nonumber\\
&=& -\int_{\R^{3}}\rho\frac{u^{r}}{r} (u^{\theta})^{2}(u^{\theta})^{2}~dx\nonumber\\
&\leq& C\|\frac{u^{r}}{r}\|_{\frac{18}{5}} \left\|u^{\theta}\right\|_{\frac{18}{5}}^{2}\left\|(u^{\theta})^{2}\right\|_{6}\nonumber\\
&\leq&C \|\frac{u^{r}}{r}\|_{2}^{\frac{1}{3}}\|\nabla\frac{u^{r}}{r}\|_{2}^{\frac{2}{3}}\|u^{\theta}\|_{2}^{\frac{2}{3}}\|\nabla u^{\theta}\|_{2}^{\frac{4}{3}}\|\nabla(u^{\theta})^{2}\|_{2}\nonumber\\
&\leq&C\| u^{\theta}\|_{2}^{\frac{2}{3}}\|\Gamma\|_{2}^{\frac{2}{3}}\|\nabla\mathbf{u}\|_{2}^{\frac{5}{3}}\left\|\nabla (u^{\theta})^{2}\right\|_{2}\nonumber\\
&\leq&C\|\mathbf{u}\|_{2}^{\frac{4}{3}}\|\Gamma\|_{2}^{\frac{4}{3}}\|\nabla \mathbf{u}\|_{2}^{\frac{10}{3}}+\frac{1}{2}\left\|\nabla (u^{\theta})^{2}\right\|_{2}^{2}.
\end{eqnarray*}
 $\hfill\Box$

\noindent\textbf{Proof of Lemma \ref{lem2.5}.}

Let $r_{0}>0$. Applying the Gagliardo-Nirenberg-Sobolev inequality, Cauchy inequality, (\ref{2.1}), (\ref{2.2}) and Lemma \ref{lem1}, we have
\begin{eqnarray*}
&&\|u_{3}\partial_{3}u^{3}\|_{2}^{2}\\
&\leq&\|u^{3}\partial_{r}u^{r}\|_{2}^{2}+\|u^{3}\frac{u^{r}}{r}\|_{2}^{2}\\
&=&\|u^{3}\partial_{r}u^{r}(|_{r\leq r_{0}}+|_{r>r_{0}})\|_{2}^{2}+\|u^{3}\frac{u^{r}}{r}\|_{2}^{2}\\
&\leq&\|u^{3}(r\partial_{r}\frac{u^{r}}{r}+\frac{u^{r}}{r})|_{r\leq r_{0}}\|_{2}^{2}+\|u^{3}\partial_{r}u^{r}|_{r>r_{0}}\|_{2}^{2}+\|u^{3}\frac{u^{r}}{r}\|_{2}^{2}\\
&\leq& r_{0}^{2}\|u^{3}\partial_{r}\frac{u^{r}}{r}\|_{2}^{2}+\|u^{3}\partial_{r}u^{r}|_{r>r_{0}}\|_{2}^{2}+2\|u^{3}\frac{u^{r}}{r}\|_{2}^{2}\\
&\leq&r_{0}^{2} \|u^{3}\|_{\infty}^{2}\|\partial_{r}\frac{u^{r}}{r}\|_{2}^{2}+\|u^{3}\partial_{r}u^{r}|_{r>r_{0}}\|_{2}^{2}+2\|u^{3}\|_{9}^{2}\|\frac{u^{r}}{r}\|_{\frac{18}{7}}^{2}\\
&\leq&C r_{0}^{2}\|u^{3}\|_{\dot{H}^{1}}\|u^{3}\|_{\dot{H}^{2}}\|\partial_{r}\frac{u^{r}}{r}\|_{2}^{2}+\|u^{3}\partial_{r}u^{r}|_{r>r_{0}}\|_{2}^{2}+C\|u^{3}\|_{\frac{18}{7}}\|\Delta u^{3}\|_{2}\|\frac{u^{r}}{r}\|_{2}^{\frac{4}{3}}\|\nabla\frac{u^{r}}{r}\|_{2}^{\frac{2}{3}}\\
&\leq&C r_{0}^{2}\|\nabla u^{3}\|_{2}\|\Delta u^{3}\|_{2}\|\Gamma\|_{2}^{2}+\|u^{3}\partial_{r}u^{r}|_{r>r_{0}}\|_{2}^{2}+C\|u^{3}\|_{2}^{\frac{2}{3}}\|\nabla u^{3}\|_{2}^{\frac{1}{3}}(\|\nabla \omega^{\theta}\|_{2}+\|\Gamma\|_{2})\|\frac{u^{r}}{r}\|_{2}^{\frac{4}{3}}\|\Gamma\|_{2}^{\frac{2}{3}}\\
&\leq&Cr_{0}^{2}\|\nabla \mathbf{u}\|_{2}\|(\|\partial_{r}\omega^{\theta}\|_{2}+\|\Gamma\|_{2})\|\Gamma\|_{2}^{2}+\|u^{3}\partial_{r}u^{r}|_{r>r_{0}}\|_{2}^{2}\\
&&+C\|\mathbf{u}\|_{2}^{\frac{2}{3}}\|\Gamma\|_{2}^{\frac{2}{3}}\|\nabla \mathbf{u}\|_{2}^{\frac{5}{3}}(\|\nabla \omega^{\theta}\|_{2}+\|\Gamma\|_{2}),
\end{eqnarray*}
\begin{eqnarray*}
\|u^{r}\partial_{r}\mathbf{\tilde u}\|_{2}^{2}& \leq&\|r\frac{u^{r}}{r}\partial_{r}\mathbf{\tilde u}(|_{r\leq r_{0}}+|_{r>r_{0}})\|_{2}^{2}\\
&\leq&r_{0}^{2}\|\frac{u^{r}}{r}\partial_{r}\mathbf{\tilde u}\|_{2}^{2}+\|u^{r}\partial_{r}\mathbf{\tilde u}|_{r>r_{0}}\|_{2}^{2}\\
&\leq&Cr_{0}^{2}\|\frac{u^{r}}{r}\|_{6}^{2}\|\partial_{r}\mathbf{\tilde u}\|_{3}^{2}+\|u^{r}\partial_{r}\mathbf{\tilde u}|_{r>r_{0}}\|_{2}^{2}\\
&\leq&Cr_{0}^{2}\|\Gamma\|_{2}^{2}\|\omega^{\theta}\|_{3}^{2}+\|u^{r}\partial_{r}\mathbf{\tilde u}|_{r>r_{0}}\|_{2}^{2}\\
&\leq&Cr_{0}^{2}\|\nabla\mathbf{u}\|_{2}\|\nabla\omega^{\theta}\|_{2}\|\Gamma\|_{2}^{2}+\|u^{r}\partial_{r}\mathbf{\tilde u}|_{r>r_{0}}\|_{2}^{2},
\end{eqnarray*}
and
\begin{eqnarray*}
\|u^{3}\partial_{3}u^{r}\|_{2}^{2}& \leq&\|u^{3}r\partial_{3}\frac{u^{r}}{r}(|_{r\leq r_{0}}+|_{r>r_{0}})\|_{2}^{2}\\
&\leq&r_{0}^{2}\|u^{3}\partial_{3}\frac{u^{r}}{r}\|_{2}^{2}+\|u^{3}\partial_{3}u^{r}|_{r>r_{0}}\|_{2}^{2}\\
&\leq&r_{0}^{2}\|u^{3}\|_{\infty}^{2}\|\partial_{3}\frac{u^{r}}{r}\|_{2}^{2}+\|u^{3}\partial_{3}u^{r}|_{r>r_{0}}\|_{2}^{2}\\
&\leq&Cr_{0}^{2}\|\nabla u^{3}\|_{2}\|\Delta u^{3}\|_{2}\|\Gamma\|_{2}^{2}+\|u^{3}\partial_{3}u^{r}|_{r>r_{0}}\|_{2}^{2}\\
&\leq&Cr_{0}^{2}\|\nabla\mathbf{u}\|_{2}\|\Gamma\|_{2}^{2}(\|\nabla\omega^{\theta}\|_{2}+\|\Gamma\|_{2})+\|u^{3}\partial_{3}u^{r}|_{r>r_{0}}\|_{2}^{2}.
\end{eqnarray*}
Applying (\ref{1.4}), (\ref{2.2}) and Lemma \ref{lem1}, we have
\begin{eqnarray}\label{5.1}
\|\tilde\nabla\tilde\nabla \mathbf{\tilde u}\|_{2} &\leq&C \|\Delta u^{3}\|_{2}+\|\tilde \nabla \partial_{3} u^{r}\|_{2}+\|\tilde \nabla \partial_{r} u^{r}\|_{2}\nonumber \\
&\leq& C\|\Delta u^{3}\|_{2}+\|\tilde \nabla (\omega^{\theta}+\partial_{r}u^{3})\|_{2}+\|\tilde \nabla (\partial_{3}u^{3}+\frac{u^{r}}{r})\|_{2}\nonumber \\
&\leq&C\|\Delta u^{3}\|_{2}+\|\tilde \nabla \omega^{\theta}\|_{2}+\|\tilde \nabla \frac{u^{r}}{r}\|_{2}\nonumber \\
&\leq&C\|\nabla \omega^{\theta}\|_{2}+C\|\Gamma\|_{2}.
\end{eqnarray}
Thus, applying the Sobolev inequality in 2-Dimension, (\ref{2.1}), (\ref{2.2}), (\ref{5.1}) and Lemma \ref{lem1}, we have
\begin{eqnarray}\label{2.12}
\|\mathbf{\tilde u}\cdot\tilde\nabla\mathbf{\tilde u}|_{r>r_{0}}\|_{2}^{2}&=&\int_{\R}\int_{r_{0}}^{\infty}|\mathbf{\tilde u}\cdot\tilde\nabla\mathbf{\tilde u}|^{2}~rdrdx_{3}\nonumber\\
&\leq&(\underset{r>r_{0},x_{3}\in\R}{\sup}|\mathbf{\tilde u}|^{2})\int_{\R}\int_{r_{0}}^{\infty}|\tilde\nabla\mathbf{\tilde u}|^{2}~rdrdx_{3}\nonumber\\
&\leq& C (\int_{\R}\int_{r_{0}}^{\infty}|\mathbf{\tilde u}|^{2}drdx_{3})^{\frac{1}{2}}(\int_{\R}\int_{r_{0}}^{\infty}|\tilde\nabla\tilde\nabla \mathbf{\tilde u}|^{2}drdx_{3})^{\frac{1}{2}}\|\nabla\mathbf{u}\|_{2}^{2}\nonumber\\
&\leq& \frac{C}{r_{0}} \|\mathbf{u}\|_{2}( \|\Gamma\|_{2}+ \|\nabla \omega^{\theta}\|_{2})\|\nabla\mathbf{ u}\|_{2}^{2}.
\end{eqnarray}
Take
\begin{equation}\label{r0}
r_{0}=\left(\frac{\|\mathbf{u}\|_{2}\|\nabla\mathbf{u}\|_{2}}{\|\Gamma\|_{2}^{2}}\right)^{\frac{1}{3}},
\end{equation}
and combing the above estimates, we have
\begin{eqnarray*}
&&\|\mathbf{\tilde u}\cdot\tilde\nabla\mathbf{\tilde u}\|_{2}^{2}\\
&\leq&C\|\mathbf{u}\|_{2}^{\frac{2}{3}}\|\Gamma\|_{2}^{\frac{2}{3}}\|\nabla \mathbf{u}\|_{2}^{\frac{5}{3}}(\|\nabla \omega^{\theta}\|_{2}+\|\Gamma\|_{2})\\
&\leq&C_{\delta}\|\mathbf{u}\|_{2}^{\frac{4}{3}}\|\Gamma\|_{2}^{\frac{4}{3}}\|\nabla \mathbf{u}\|_{2}^{\frac{10}{3}}
+\delta(\|\nabla\omega^{\theta}\|_{2}^{2}+\|\Gamma\|_{2}^{2}).
\end{eqnarray*}
 $\hfill\Box$

 \noindent\textbf{Proof of Lemma \ref{l3.1}.}

It is easy to see that $m\leq J^{\epsilon}\ast\rho_{0} \leq M,$ and
$$
|(J^{\epsilon}\ast(\rho_{0}-1))(0,x_{3})|=\int J^{\epsilon}(r,z) |\rho_{0}(r,x_{3}-z)-1| ~rdrdz\leq C_{0}\int J^{\epsilon}~ r~dx\leq C\epsilon .
$$
Then $0<\frac{m}{2}\leq\rho_{0}^{\epsilon}\leq \frac{M}{2}$ when $\epsilon$ is sufficient small.

Set $r=\sqrt{x_{1}^{2}+x_{2}^{2}}$, $s=\sqrt{y_{1}^{2}+y_{2}^{2}}$ and $C_{0}=\|\frac{\rho_{0}-1}{r}\|_{\infty}.$
When $r\leq 5\varepsilon$, it can be deduced directly,
\begin{eqnarray*}
|\rho_{0}^{\epsilon}(x)-1|&=& |J^{\epsilon}\ast(\rho_{0}-1)-(J^{\epsilon}\ast(\rho_{0}-1))(0,x_{3})|\\
&\leq&|\int (J^{\epsilon}(x-y)-J^{\epsilon}(s,x_{3}-y_{3}))(\rho_{0}(y)-1) ~dy|\\
&\leq&C_{0} \int |J^{\epsilon}(\sqrt{r^{2}+s^{2}-2rs\cos(\tau)},x_{3}-y_{3})-J^{\epsilon}(s,x_{3}-y_{3})| s^2dsd\tau dy_{3}\\
&\leq&C_{0} \int_{s\leq 6\epsilon,|x_3-y_{3}|\leq\epsilon}|\sqrt{r^{2}+s^{2}-2rs\cos(\tau)}-s||\nabla J^{\epsilon}|_{\infty} s^2dsd\tau dy_{3}\\
&\leq&C_{0} C r\int_{s\leq 6\epsilon,|x_3-y_{3}|\leq\epsilon}\epsilon^{-4} sdy \\
&\leq&C_{0} C r .
\end{eqnarray*}
When $r> 5\varepsilon$, we have that
    $
    s\in[\frac{r}{2},2r]$,   if $ (x-y)\in \mathrm{supp} J^\epsilon$,
and
\begin{eqnarray*}
|\rho_{0}^{\epsilon}(x)-1|&=& |J^{\epsilon}\ast(\rho_{0}-1)-(J^{\epsilon}\ast(\rho_{0}-1))(0,x_{3})|\\
&\leq&C_0\int J^{\epsilon}(x-y)s dy+C_0\int J^{\epsilon}(s,x_{3}-y_{3}) s dy\\
&\leq&2C_{0}r+C_0\epsilon \\
&\leq& CC_0r.
\end{eqnarray*}
 $\hfill\Box$

\section*{Acknowledgements}
  This work is
  partially supported by  NSF of
China under Grants 11271322,  11331005 and 11271017, National Program for
Special Support of Top-Notch Young Professionals.

\end{document}